\documentclass[oneside,english]{article}
\usepackage[T1]{fontenc}
\usepackage[latin9]{inputenc}
\usepackage{float}
\usepackage{epic,eepic}
\usepackage{comment}
\usepackage{amstext,amsthm,amssymb,amsmath}
\usepackage{epsfig,subfigure,epstopdf}
\usepackage{graphicx}
\usepackage{subfigure}
\usepackage{wrapfig}
\usepackage{fullpage}
\usepackage{color}
\usepackage{multirow}
\usepackage{tabulary}
\usepackage{booktabs}
\usepackage{enumerate}
\usepackage{color}

\theoremstyle{definition}
\newtheorem{theorem}{Theorem}[section]
\newtheorem{lemma}{Lemma}[section]

\newtheorem{example}{Example}

\makeatletter
\numberwithin{equation}{section}
\numberwithin{figure}{section}

\makeatother

\usepackage{babel}

\title{Partially Explicit Time Discretization for Nonlinear Time Fractional Diffusion Equations}

\author{Wenyuan Li\footnote{Department of Mathematics, Texas A\&M University, College Station, TX 77843, USA. (\texttt{E-mail: wenyuanli@tamu.edu})}, ~Anatoly Alikhanov\footnote{North-Caucasus Center for Mathematical Research, North-Caucasus Federal University,  Stavropol, 355017  Russia. (\texttt{E-mail: aaalikhanov@gmail.com})}, ~Yalchin Efendiev\footnote{Department of Mathematics, Texas A\&M University, College Station, TX 77843, USA \& North-Eastern Federal University, Yakutsk, Russia. (\texttt{E-mail: efendiev@math.tamu.edu})}, ~ Wing Tat Leung\footnote{Department of Mathematics, University of California, Irvine, USA. (\texttt{E-mail: wtleung@uci.edu})}}
\begin{document}
\maketitle

\section*{Abstract}
Nonlinear time fractional partial differential equations are widely used in modeling and simulations. In many applications, there are high contrast changes in media properties. For solving these problems, one often uses coarse spatial
grid for spatial resolution. For temporal discretization, implicit methods
are often used.
For implicit methods, though the time step can be relatively large, 
the equations are difficult to compute due to the nonlinearity and 
the fact that one deals with large-scale systems. 
On the other hand, the discrete system in explicit methods are easier to 
compute but it requires small time steps. 
In this work, we propose the partially explicit scheme following 
earlier works on developing partially explicit methods for nonlinear diffusion equations. In this scheme, the diffusion term is treated partially explicitly and the reaction term is treated fully explicitly. With the appropriate construction of spaces and stability analysis, we find that the required time step in our proposed scheme scales as the coarse mesh size, which creates a great saving in computing.  The main novelty of this work is the extension
of our earlier works for diffusion equations to time fractional diffusion
equations. For the case of fractional diffusion equations, the constraints on
time steps are more severe and the proposed methods alleviate this
since the time step in partially explicit method scales as the coarse mesh size.
We present stability results. Numerical results are presented where
we compare our proposed partially explicit methods with a fully implicit 
approach. We show that our proposed approach provides similar results, while
treating many degrees of freedom in nonlinear terms explicitly.

\section{Introduction}

Multiscale features arise in many problems such as the flow in porous media. The multiscale nature appears because of the variations in media properties.  For instance, the permeability field can vary at different scales and have high contrast. There are many multiscale methods developed for parabolic partial differential equations. However, there are only a small amount of research on multiscale time fractional diffusion reaction equations while they have many applications.

Because of the successful applications of fractional calculus in many different fields \cite{Kilbas2006TheoryAA,Hilfer2000ApplicationsOF,Oldham1974TheFC,Podlubny1999FractionalDE}, equations with fractional derivatives are discussed in many recent research. 
Fractional derivatives are used for describing processes that  lead
 to  equations of fractional orders.
Many approaches have been put forward to numerically solve the time fractional partial differential equations. The partially explicit and splitting scheme for linear time fractional partial differential equations are developed in \cite{hu2021partially}. This paper is a continuation of the linear case and designs the partially explicit scheme for time fractional nonlinear diffusion reaction equations.

In previous works, there are many developed spatial approaches for multiscale linear and nonlinear partial differential equations. For linear problems, many multiscale algorithms have been put forward. 
These include
homogenization-based approaches \cite{eh09,le2014msfem}, 
multiscale
finite element methods \cite{eh09,hw97,jennylt03}, 
generalized multiscale finite element methods (GMsFEM) \cite{chung2016adaptiveJCP,MixedGMsFEM,WaveGMsFEM,chung2018fast,GMsFEM13}, 
constraint energy minimizing GMsFEM (CEM-GMsFEM) \cite{chung2018constraint, chung2018constraintmixed}, nonlocal
multi-continua (NLMC) approaches \cite{NLMC},
metric-based upscaling \cite{oz06_1}, heterogeneous multiscale method 
\cite{ee03}, localized orthogonal decomposition (LOD) 
\cite{henning2012localized}, equation-free approaches \cite{rk07,skr06}, 
multiscale stochastic approaches \cite{hou2017exploring, hou2019model, hou2018adaptive},
and hierarchical multiscale method \cite{brown2013efficient}.
Algorithms such as GMsFEM \cite{chung2018constraint} and NLMC are developed for high-contrast problems to capture the multiscale features
\cite{chung2018constraintmixed, NLMC}. 
These algorithms need a careful construction of multiscale
dominant modes. For nonlinear problems, we can use nonlinear maps to replace linear multiscale basis functions \cite{ep03a,ep03c,efendiev2014generalized}.

Splitting methods for nonlinear PDEs are extensively studied in the literature. Earlier approaches include implicit-explicit approaches and other techniques
\cite{ascher1997implicit,li2008effectiveness,abdulle2012explicit,engquist2005heterogeneous,ariel2009multiscale,narayanamurthi2019epirk,shi2019local,duchemin2014explicit,frank1997stability,izzo2017highly,ruuth1995implicit,hundsdorfer2007imex,du2019third}.
In recent works, splitting methods are designed to solve the linear and nonlinear parabolic partial differential equations \cite{chung_partial_expliict21,chung2021contrastindependent}. 
This paper applies the splitting method to time fractional partial differential equations.

Our work starts with the nonlinear time fractional partial differential 
equation
\[ \cfrac{\partial^{\alpha} u}{\partial t^{\alpha}} + f(u) + g(u) = 0,   \]
where $f(u)$ is the diffusion operator and $g(u)$ is the reaction term. $f(u)$ is contrast dependent while $g(u)$ is not and represents reaction terms. Based on properties of $f(u)$ and $g(u)$, we formulate a condition for time step so that our proposed partially explicit scheme is stable. This condition is similar to that of the nonlinear non-fractional PDE \cite{chung2021contrastindependent}. By careful construction of spaces, the stability condition is contrast-independent \cite{chung_partial_expliict21}. We also find that the restrictive time step scales as the coarse mesh size and thus can be much larger.
Our proposed approach provides explicit treatment for nonlinear
terms ($f(u)$ and $g(u)$), while we still have implicit 
coupling in time derivative part of the discrete system. The
coupling in time derivative
can be removed by designing some mass lumping (see \cite{chung2021contrast}).
However, the design of mass lumping for time fractional diffusion is
challenging and we leave it for future studies. In general, one can
easily solve the coupled mass matrix and the main difficulty is
in handling high contrast diffusion equation, which we handle via 
our partial explicit method.

The main findings of our paper are as follows.
\begin{itemize}

\item We propose a partially explicit scheme for time fractional
nonlinear equations.

\item The time step in our method scales as the coarse mesh size.

\item The stability analysis is performed for time fractional case.

\end{itemize}

We present numerical results. We consider high contrast linear
and nonlinear diffusion operators
and nonlinear reaction terms. We identify coarse degrees of freedom
and solve them in nonlinear systems implicitly. Other degrees of freedom
are treated explicitly. We consider two types of source terms.
The first source term is a smooth function and the second source term
is a singular functions with locally supported jump. In all cases,
we compare our proposed method with the fully implicit method. 
We show that our proposed method gives similar accuracy while
treating additional degrees of freedom in nonlinear terms explicitly.

This paper is organized as follows. In the next section, we present the problem setting. In Section 3, we provide three schemes and the corresponding stability analysis. In Section 4, we discuss the construction of spaces and present numerical results.

\section{Problem Setting}

Let $\Omega$ be a bounded domain in $\mathbb{R}^d$ ($d=1,2,3$) and the boundary $\partial \Omega$ is sufficiently smooth.
We want to discuss the 
time fractional nonlinear diffusion reaction equation
\begin{align}
    \frac{\partial^{\alpha} u(x,t) }{\partial t^{\alpha}} + f(u) + g(u)
      = 0.
      \label{eq0}
\end{align} 

We let $t_0 = 0$ be the initial time and $T$ be the final time.
We define $\Delta t = \frac{T}{K}$ ($\Delta t << T $), $t_j = j\cdot \Delta t$ and $u^j = u(t_j)$, $0\leq j \leq K$.
Let $b_k = (k+1)^{1-\alpha} - k^{1-\alpha}$ and 
$\alpha_0 = \Gamma(2-\alpha) (\Delta t)^{\alpha}$.
Discretize the time fractional derivative as in \cite{LIN20071533}, we have
\begin{align*}
\frac{\partial^{\alpha} u(x,t_{k+1})}{\partial t^{\alpha}} 
=& \frac{1}{\Gamma(1-\alpha)} \int_0^{t_{k+1}} \frac{\partial u(x,s)}{\partial s} \frac{1}{(t_{k+1}-s)^{\alpha}} ds \\
=& \frac{1}{\Gamma (1-\alpha)} \sum_{j=0}^k \int^{t_{j+1}}_{t_{j}} 
   \frac{\partial u(x,s)}{\partial s} \frac{1}{(t_{k+1}-s)^{\alpha}} ds\\
\approx & \frac{1}{\Gamma (1-\alpha)} \sum_{j=0}^k 
  \frac{u(t_{j+1}) - u(t_j)}{\Delta t} \int^{t_{j+1}}_{t_{j}} 
    \frac{1}{(t_{k+1}-s)^{\alpha}} ds \\
=& \frac{1}{\Gamma(1-\alpha) \Delta t}  \sum_{j=0}^k (u(t_{j+1}) - u(t_j)) \Big(\frac{-1}{1-\alpha}    (t_{k+1}-s)^{1-\alpha} \Big) \Big|_{t_j}^{t_{j+1}} \\
=& \frac{1}{(1-\alpha)\Gamma(1-\alpha)\Delta t} \sum_{j=0}^k 
 (u(t_{j+1}) - u(t_j)) ((t_{k+1}-t_j)^{1-\alpha} - (t_{k+1}-t_{j+1})^{1-\alpha})\\
=& \frac{1}{\Gamma(2-\alpha)\Delta t} 
\sum_{j=0}^k (u(t_{j+1}) - u(t_j)) (\Delta t)^{1-\alpha}
((k+1-j)^{1-\alpha} - (k-j)^{1-\alpha} )\\
=& \frac{1}{\Gamma(2-\alpha)(\Delta t)^{\alpha}} 
\sum_{j=0}^k b_{k-j} (u(t_{j+1}) - u(t_j))\\
=&\frac{1}{\alpha_0} \sum_{j=0}^k b_{k-j} (u^{j+1} - u^j).
\end{align*}

Now Equation (~\ref{eq0}) becomes 
\begin{align*}
    \frac{1}{\alpha_0} \sum_{j=0}^k b_{k-j} (u^{j+1} - u^j) 
    + f(u) + g(u) = 0,
\end{align*}
where $f = \cfrac{\delta F}{\delta u}$ and $g = \cfrac{\delta G}{\delta u}$.
$f$ is the contrast dependent term (which introduces stiffness) while $g$ does not depend on the contrast. 
Let $V$ be a Hilbert space.
We assume that $f(u)\in V^{\ast}$ and $g(u) \in L^2(\Omega)$ for all $u\in V$.
For simplicity, we use $(\cdot,\cdot)$ to denote $(\cdot,\cdot)_{V^{\ast},V}$.
To prove the stability of our schemes, we need to make some assumptions on $F$ and $G$.  We use $\|\cdot\|$ to denote $\|\cdot\|_{L^2}$.
\begin{itemize}
    \item {The second variational derivative of $F$ and $G$ satisfy} 
    \begin{align*}
    \delta^2 F(u) ( v, v) \geq c(u)\|v\|_V^2, \quad \forall u , v \in V, \\
    \delta^2 G(u) ( v, v) \geq b(u)\|v\|^2, \quad \forall u , v \in V,
\end{align*}
where $0 \leq c(u) < \infty$ and $-\underline{b} \leq b(u) < \infty$ 
($\underline{b}>0$) do not depend on $v$.
\item {The second variational derivative of $F$ and $G$ is bounded}
\begin{align*}
    |\delta^2 F(u) ( w,v )| \leq C(u) \|v\|_V \|w\|_V, \quad \forall u,v,w \in V,\\
    |\delta^2 G(u) ( w,v )| \leq B \|v\| \|w\|, \quad \forall u,v,w \in V,
\end{align*}
where $0<C(u)<\infty$ and $0<B<\infty$ do not depend on $v$, $w$.
\end{itemize}
\begin{example}
For $ F = \cfrac{1}{2} \int_{\Omega} \kappa |\nabla u |^2$, we have 
\[ \left( \cfrac{\delta F}{ \delta u } , v \right) = 
    \int_{\Omega} \kappa \nabla u \cdot \nabla v \]
    and 
    \[ \delta^2 F(u) (w,v) = \int_{\Omega} \kappa \nabla v \cdot \nabla w, \quad \forall u \in V.  \]
    In this case, the norm we use is $ \|v\|_V = \| \kappa^{\frac{1}{2}} \nabla v \|$. We find that 
    \[ c(u) = C(u) \equiv 1. \]
\end{example}
We also need the following lemma to prove stability (see \cite{hu2021partially}).
\begin{lemma} 
Let $N\in \mathbb{N}$. We have
\begin{align*}
    \sum_{k=0}^N \sum_{j=0}^k b_{k-j} (u^{j+1}-u^j,u^{k+1}-u^k )
    \geq \frac{1}{2} \sum_{k=0}^N \|u^{k+1}-u^k\|^2.
\end{align*}
\label{lem1}
\end{lemma}
\begin{proof}
By definition of $b_k$, we have
\begin{align*}
    b_k = (k+1)^{1-\alpha} - k^{1-\alpha} 
        = \frac{1}{1-\alpha} \int_0^1 (k+s)^{-\alpha} ds. 
\end{align*}
Note that $\theta(k) = (k+s)^{-\alpha}$ is a complete monotonic function. 
By Hausdorff-Bernstein-Widder theorem, we obtain 
\begin{align*}
    (k+s)^{-\alpha} = \int_0^{\infty} e^{-k\tau } dg_s(\tau), 
\end{align*}
where $g_s$ is a cumulative distribution function.
Thus, we have 
\begin{align*}
    &\quad \sum_{k=0}^N \sum_{j=0}^N b_{|k-j|}(u^{j+1}-u^j,u^{k+1}-u^k) \\
    &= \frac{1}{1-\alpha} \int_{\Omega} \int_0^1 \int_0^{\infty} 
    \sum_{k=0}^N \sum_{j=0}^N e^{-|k-j|\tau}(u^{j+1}-u^j)(u^{k+1}-u^k) dg_s(\tau)ds dx.
\end{align*}
Let $M$ be a matrix with $M_{k,j} = e^{-|k-j| t}$. Then $M$ is positive definite for $t>0$. Thus, we get
\begin{align*}
    \sum_{k=0}^N \sum_{j=0}^N b_{|k-j|}(u^{j+1}-u^j,u^{k+1}-u^k) >0.
\end{align*}
With this inequality, we find that
\begin{align*}
    &\sum_{k=0}^N \sum_{j=0}^k b_{k-j} (u^{j+1}-u^j,u^{k+1}-u^k ) \\
    = & \frac{1}{2} \sum_{k=0}^N (u^{k+1}-u^k, u^{k+1}-u^k) 
    + \frac{1}{2} \sum_{k=0}^N \sum_{j=0}^N b_{|k-j|}(u^{j+1}-u^j,u^{k+1}-u^k)\\
    \geq & \frac{1}{2} \sum_{k=0}^N (u^{k+1}-u^k, u^{k+1}-u^k)
    = \frac{1}{2} \sum_{k=0}^N \|u^{k+1}-u^k\|^2.
\end{align*}
\end{proof}

\section{Three Schemes and Stability}
\subsection{Implicit Scheme}
Implicit scheme is also called backward Euler scheme. 
Let $V_H \subset H^1_0(\Omega)$ be the finite element space.
In this case, we consider $\{u_k\}_{k=0}^K \subset V_H$ such that:
\begin{align}
    ( \frac{1}{\alpha_0} \sum_{j=0}^k b_{k-j} ( u^{j+1}-u^j ), v  )
     + ( f(u^{k+1}) + g(u^{k+1}) , v ) = 0, \quad\forall v\in V_H.
     \label{ImplicitSys}
\end{align}
Now we want to show the stability of this implicit case.
\begin{theorem}
If $\alpha_0 \leq \cfrac{1}{\underline{b}}$, 
then the implicit scheme (\ref{ImplicitSys}) is stable with
\[ F(u^{N+1}) + G(u^{N+1}) \leq F(u^{0}) + G(u^{0}), \]
where $0\leq N \leq K-1$ and $N\in\mathbb{N}$.
\end{theorem}
\begin{proof}
Let $v = u^{k+1}-u^k $.
As \begin{align*}
    F(u^{k+1}) = F(u^k) + (f(u^{k+1}) , u^{k+1}-u^k) 
    - \frac{1}{2} \delta^2 F(\xi_1^k) (u^{k+1}-u^k, u^{k+1}-u^k),\\
    G(u^{k+1}) = G(u^k) + (g(u^{k+1}) , u^{k+1}-u^k) 
    - \frac{1}{2} \delta^2 G(\xi_2^k) (u^{k+1}-u^k, u^{k+1}-u^k),
\end{align*}
where $\xi_i^k = \lambda_i^k u^k + (1-\lambda_i^k) u^{k+1}$ with 
$\lambda_i^k \in [0,1]$ and $i\in\{1,2\}$. By Lemma \ref{lem1} and assumptions on $F$ and $G$, we have 
\allowdisplaybreaks
\begin{align*}
    0 &= \frac{1}{\alpha_0} \sum_{k=0}^N \sum_{j=0}^k b_{k-j} (u^{j+1}-u^j,u^{k+1}-u^k ) + \sum_{k=0}^N ( f(u^{k+1}) + g(u^{k+1}) , u^{k+1}-u^k)\\
     &= \frac{1}{\alpha_0} \sum_{k=0}^N \sum_{j=0}^k b_{k-j} 
     (u^{j+1}-u^j,u^{k+1}-u^k )
     + \sum_{k=0}^N [F(u^{k+1}) - F(u^k) + \frac{1}{2} 
     \delta^2 F(\xi_1^k) (u^{k+1}-u^k,u^{k+1}-u^k) \\
      & \quad\quad\quad\quad\quad\quad\quad\quad\quad\quad\quad
      \quad\quad\quad\quad\quad\quad\quad\quad\quad
      + G(u^{k+1}) - G(u^k) + \frac{1}{2}\delta^2 G(\xi_2^k)
      (u^{k+1}-u^k,u^{k+1}-u^k) ]\\
      &\geq \frac{1}{2\alpha_0} \sum_{k=0}^N \|u^{k+1}-u^k\|^2 
      + \sum_{k=0}^N [F(u^{k+1}) - F(u^k) + 
      \frac{c(\xi_1^k)}{2} \|u^{k+1}-u^k\|_V^2 \\
      & \quad\quad\quad\quad\quad\quad\quad\quad\quad\quad\quad\quad
      + G(u^{k+1}) - G(u^k) + \frac{b(\xi_2^k)}{2}\|u^{k+1}-u^k\|^2 ]\\
      &\geq (\frac{1}{2\alpha_0} - \frac{\underline{b}}{2} )
      \sum_{k=0}^N \|u^{k+1}-u^k\|^2 + 
      F(u^{N+1}) - F(u^0) + G(u^{N+1}) - G(u^0)+0.
\end{align*}

Thus, if $\alpha_0 \underline{b} \leq 1$, we have
\begin{align*}
    F(u^{N+1}) + G(u^{N+1}) \leq F(u^{0}) + G(u^{0}).
\end{align*}
\end{proof}

\subsection{Explicit Scheme}
We consider $\{u^k\}_{k=0}^K\subset V_H$ such that
\begin{align}
    ( \frac{1}{\alpha_0} \sum_{j=0}^k b_{k-j} ( u^{j+1}-u^j ), v  )
     + ( f(u^{k}) + g(u^{k}) , v ) = 0,\quad\forall v\in V_H.
     \label{ExplicitSys}
\end{align}
\begin{theorem}
If $\alpha_0 \Big(\cfrac{C(\xi) \|u^{k+1}-u^k\|^2_V }{\|u^{k+1}-u^k\|^2} + B \Big) \leq 1$ for any $\xi = (1-\lambda) u^{k+1} + \lambda u^k $ with $\lambda \in [0,1]$ and $0 \leq k\leq N$, then the explicit scheme (\ref{ExplicitSys}) is stable with
\[ F(u^{N+1}) + G(u^{N+1}) \leq F(u^{0}) + G(u^{0}), \]
where $0\leq N \leq K-1$ and $N\in\mathbb{N}$.
\end{theorem}
\begin{proof}
Let $v = u^{k+1}-u^k $. 
As \begin{align*}
    F(u^{k+1}) = F(u^k) + (f(u^{k}) , u^{k+1}-u^k) 
    + \frac{1}{2} \delta^2 F(\xi_1^k) (u^{k+1}-u^k, u^{k+1}-u^k),\\
    G(u^{k+1}) = G(u^k) + (g(u^{k}) , u^{k+1}-u^k) 
    + \frac{1}{2} \delta^2 G(\xi_2^k) (u^{k+1}-u^k, u^{k+1}-u^k),
\end{align*}
where $\xi_i^k = \lambda_i^k u^k + (1-\lambda_i^k) u^{k+1}$ with 
$\lambda_i^k \in [0,1]$ and $i\in\{1,2\}$.
By Lemma ~\ref{lem1} and assumptions on $F$ and $G$, we have
\begin{align*}
    0 &= \frac{1}{\alpha_0} \sum_{k=0}^N \sum_{j=0}^k b_{k-j} ( u^{j+1}-u^j , u^{k+1}-u^k ) + \sum_{k=0}^N ( f(u^k) + g(u^k) , u^{k+1}-u^k ) \\
    &= \frac{1}{\alpha_0} \sum_{k=0}^N \sum_{j=0}^k b_{k-j} ( u^{j+1}-u^j , u^{k+1}-u^k ) + \sum_{k=0}^N \Big[ F( u^{k+1}) - F(u^k) - \frac{1}{2} 
     \delta^2 F(\xi_1^k) (u^{k+1}-u^k,u^{k+1}-u^k) \\
      & \quad\quad\quad\quad\quad\quad\quad\quad\quad\quad\quad
      \quad\quad\quad\quad\quad\quad\quad\quad\quad
      + G(u^{k+1}) - G(u^k) - \frac{1}{2}\delta^2 G(\xi_2^k)
      (u^{k+1}-u^k,u^{k+1}-u^k) \Big]\\
    &\geq \frac{1}{2\alpha_0}\sum_{k=0}^N \| u^{k+1}-u^k \|^2 
    + \sum_{k=0}^N \Big[ F(u^{k+1}) - F(u^k) -\frac{C(\xi_1^k)}{2} \|u^{k+1}-u^k\|^2_V\\
      & \quad\quad\quad\quad\quad\quad\quad\quad\quad\quad\quad\quad\quad
      + G(u^{k+1}) - G(u^k) - \frac{B}{2} \|u^{k+1}-u^k\|^2 \Big]\\
    &= \sum_{k=0}^N \Big[( \frac{1}{2\alpha_0} - \frac{B}{2}) \|u^{k+1}-u^k\|^2       -\frac{C(\xi^k_1)}{2} \|u^{k+1}-u^k\|^2_V \Big] 
       + F(u^{N+1}) + G(u^{N+1}) 
       - F(u^{0}) + G(u^{0}).
\end{align*}
Thus, if $ \alpha_0 \Big(\cfrac{C(\xi) \|u^{k+1}-u^k\|^2_V }{\|u^{k+1}-u^k\|^2} + B \Big) \leq 1$ for any $\xi = (1-\lambda) u^{k+1} + \lambda u^k $ with $0 \leq k\leq N$, we have
\[ F(u^{N+1}) + G(u^{N+1}) \leq  F(u^{0}) + G(u^{0}).  \]
\end{proof}
We remark that when the contrast is large, the time step in explicit scheme must be extremely small.

\subsection{Partially Explicit Scheme}
To find a more efficient method, we propose the partially explicit scheme. We split $V_H$ into a direct sum of two subspace $V_{H,1}$ and $V_{H,2}$, i.e. $V_H= V_{H,1} \oplus V_{H,2} $. Then the finite element solution is $u = u_1 + u_2$ with $u_1 \in V_{H,1}$ and $u_2 \in V_{H,2}$. We will discuss the construction of these two space in Section \ref{SpaceConstructSection}. Consider $\{u^k\}_{k=0}^K \subset V_H$ such that
\begin{align*}
    (\frac{1}{\alpha_0} \sum_{j=0}^k b_{k-j} ( u^{j+1}_1-u^j_1 + u^{j+1}_2 - u^j_2 ) , v_1) 
    + (f(u^{k+1}_1 + u^k_2 ) , v_1 ) + (g(u^{k}_1 + u^k_2 ) , v_1 ) = 0, \quad 
    \forall v_1 \in V_{H,1}, \\
    (\frac{1}{\alpha_0} \sum_{j=0}^k b_{k-j} ( u^{j+1}_1-u^j_1 + u^{j+1}_2 - u^j_2 ) , v_2) 
    + (f(u^{k+1}_1 + u^k_2 ) , v_2 ) + (g(u^{k}_1 + u^k_2 ) , v_2 ) = 0, \quad 
    \forall v_2 \in V_{H,2}.
\end{align*}

As $V_{H,1} \cap V_{H,2} = \emptyset$, by the strengthened Cauchy-Schwarz inequality \cite{aldaz2013strengthened}, we have
\[ \gamma := \sup\limits_{v_1 \in V_{H,1}, v_2 \in V_{H,2} } \cfrac{|(v_1,v_2)|}{\|v_1\|\|v_2\|} <1. \]
\begin{theorem}
Let $\bar{c} = \inf_{u\in V_H } c(u)$, $\bar{C}_{2}=\sup_{\xi \in V_{H}}C_{2}(\xi)$
and 
\[
C_{2}(\xi)=\sup_{v\in V_{H},w\in V_{H,2}}\cfrac{1}{\|v\|_{V}\|w\|_{V}}\delta^{2}F(\xi)(v,w)\leq C(\xi).
\]
If \begin{align}
    \cfrac{\bar{C}_2}{\bar{c}} \sup\limits_{v_2\in V_{H,2} } \cfrac{\|v_2\|_V^2}{\|v_2\|^2} 
    \leq \cfrac{1-\gamma}{\alpha_0}-B(1+\gamma) ,
    \label{eq:stab_cond}
\end{align} 
we have 
\begin{align*}
    F(u^{N+1}) + G(u^{N+1}) 
    \leq  F(u^{0}) + G(u^{0}),
\end{align*}
where $0\leq N \leq K-1$ and $N\in \mathbb{N}$.
\end{theorem}
\begin{proof}
Let $v_1 = u_1^{k+1}-u_1^k$ and $v_2 = u_2^{k+1}-u_2^k$. We obtain
\begin{align*}
    (\frac{1}{\alpha_0} \sum_{j=0}^k b_{k-j} ( u^{j+1}_1-u^j_1 + u^{j+1}_2 - u^j_2 ) , u_1^{k+1}-u_1^k) 
    + (f(u^{k+1}_1 + u^k_2 ) , u_1^{k+1}-u_1^k ) + (g(u^{k}_1 + u^k_2 ) , u_1^{k+1}-u_1^k ) = 0, \\
    (\frac{1}{\alpha_0} \sum_{j=0}^k b_{k-j} ( u^{j+1}_1-u^j_1 + u^{j+1}_2 - u^j_2 ) , u_2^{k+1}-u_2^k) 
    + (f(u^{k+1}_1 + u^k_2 ) , u_2^{k+1}-u_2^k ) + (g(u^{k}_1 + u^k_2 ) , u_2^{k+1}-u_2^k ) = 0.
\end{align*}

Summing up these two equations, we get 
\begin{align*}
    (\frac{1}{\alpha_0} \sum_{j=0}^k b_{k-j} ( u^{j+1}-u^j) , u^{k+1}-u^k) 
    + (f(u^{k+1}_1 + u^k_2 ) , u^{k+1}-u^k ) 
    + (g(u^{k}_1 + u^k_2 ) , u^{k+1}-u^k ) = 0 .
\end{align*}
By Lemma \ref{lem1}, we have
\begingroup
\allowdisplaybreaks
\begin{align*}
    &(\frac{1}{\alpha_0} \sum_{k=0}^N \sum_{j=0}^k b_{k-j} ( u^{j+1}-u^j) , u^{k+1}-u^k)\\
    \geq & \frac{1}{2\alpha_0} \sum_{k=0}^N \|u^{k+1}-u^k\|^2\\
    =& \frac{1}{2\alpha_0} \sum_{k=0}^N \sum_{i=1}^2 \|u_i^{k+1}-u_i^k\|^2 + \frac{1}{2\alpha_0} \sum_{k=0}^N 2(u_1^{k+1}-u_1^k, u_2^{k+1}-u_2^k) \\
    \geq&  \frac{1}{2\alpha_0} \sum_{k=0}^N \sum_{i=1}^2 \|u_i^{k+1}-u_i^k\|^2 - \frac{1}{2\alpha_0} \sum_{k=0}^N 2\gamma \|u_1^{k+1}-u_1^k\| \| u_2^{k+1}-u_2^k\| \\
    \geq &  \frac{1}{2\alpha_0} \sum_{k=0}^N \sum_{i=1}^2 \|u_i^{k+1}-u_i^k\|^2 - \frac{\gamma}{2\alpha_0} \sum_{k=0}^N  \sum_{i=1}^2\|u_i^{k+1}-u_i^k\|^2 \\
    = &\frac{1-\gamma}{2\alpha_0} \sum_{k=0}^N \sum_{i=1}^2 \|u_i^{k+1}-u_i^k\|^2.
\end{align*}
\endgroup

There exists $ \xi_i^k = \lambda_i^k u^k + (1-\lambda_i^k) u^{k+1}$, $\lambda_i^k \in [0,1]$ and $i\in \{1,2\}$ such that
\begin{align*}
    (f(u^{k+1}), u^{k+1}-u^k) =  F(u^{k+1})-F(u^k) 
    + \frac{1}{2} \delta^2 F(\xi_1^k) (u^{k+1}-u^k, u^{k+1}-u^k), \\
    (g(u^k) , u^{k+1}-u^k) =  G(u^{k+1})-G(u^k) 
    - \frac{1}{2} \delta^2 G(\xi_2^k) (u^{k+1}-u^k, u^{k+1}-u^k).
\end{align*}

Then,
\begingroup
\allowdisplaybreaks
\begin{align*}
    0 =& (\frac{1}{\alpha_0} \sum_{k=0}^N \sum_{j=0}^k b_{k-j} ( u^{j+1}-u^j) , u^{k+1}-u^k)  
    + \sum_{k=0}^N (f(u_1^{k+1} +u_2^k), u^{k+1}-u^k) 
    + \sum_{k=0}^N (g(u_1^{k} +u_2^k), u^{k+1}-u^k)\\
    \geq& \frac{1-\gamma}{2\alpha_0} \sum_{k=0}^N \sum_{i=1}^2\|u_i^{k+1}-u_i^k\|^2
    + \sum_{k=0}^N (f(u_1^{k+1}+u_2^k) - f(u^{k+1}) , u^{k+1}-u^k) \\
    &\quad+ \sum_{k=0}^N (f(u^{k+1}) , u^{k+1}-u^k) 
    + \sum_{k=0}^N (g(u^k) , u^{k+1}-u^k)\\
    =& \frac{1-\gamma}{2\alpha_0} \sum_{k=0}^N \sum_{i=1}^2\|u_i^{k+1}-u_i^k\|^2 
    + \sum_{k=0}^N (f(u_1^{k+1}+u_2^k) - f(u^{k+1}) , u^{k+1}-u^k)\\
    &\quad+ \sum_{k=0}^N ( F(u^{k+1})-F(u^k) 
    + \frac{1}{2} \delta^2 F(\xi_1^k) (u^{k+1}-u^k, u^{k+1}-u^k)\\
    &\quad+ \sum_{k=0}^N ( G(u^{k+1})-G(u^k) 
    - \frac{1}{2} \delta^2 G(\xi_2^k) (u^{k+1}-u^k, u^{k+1}-u^k)\\
    \geq& \frac{1-\gamma}{2\alpha_0} \sum_{k=0}^N \sum_{i=1}^2\|u_i^{k+1}-u_i^k\|^2
    + \sum_{k=0}^N (f(u_1^{k+1}+u_2^k) - f(u^{k+1}) , u^{k+1}-u^k) \\
    &\quad+ (F(u^{N+1}) + G(u^{N+1})) 
    - ( F(u^{0}) + G(u^{0}) ) 
    + \sum_{k=0}^N \frac{c(\xi_1^k)}{2} \|u^{k+1}-u^k\|_V^2 
    - \sum_{k=0}^N  \frac{B}{2} \|u^{k+1}-u^k\|^2 \\
    \geq& (\frac{1-\gamma}{2\alpha_0}-\frac{B}{2}(1+\gamma)) \sum_{k=0}^N \sum_{i=1}^2\|u_i^{k+1}-u_i^k\|^2
    + \sum_{k=0}^N (f(u_1^{k+1}+u_2^k) - f(u^{k+1}) , u^{k+1}-u^k)\\
    &\quad + (F(u^{N+1}) + G(u^{N+1})) - ( F(u^{0}) + G(u^{0}) ) 
    + \sum_{k=0}^N \frac{\bar{c}}{2} \|u^{k+1}-u^k\|_V^2 .
\end{align*}
\endgroup

It suffices to show 
\begin{align*}
    \sum_{k=0}^N ( f(u^{k+1})-f(u_1^{k+1}+u_2^k), u^{k+1}-u^k) 
    \leq (\frac{1-\gamma}{2\alpha_0}-\frac{B}{2}(1+\gamma)) \sum_{k=0}^N \sum_{i=1}^2\|u_i^{k+1}-u_i^k\|^2
    + \sum_{k=0}^N \frac{\bar{c}}{2} \|u^{k+1}-u^k\|_V^2 .
\end{align*}

There exists $\tilde{\xi}^k = \tilde{\lambda}^k (u_1^{k+1} + u_2^k) + (1-\tilde{\lambda}^{k+1}) u^{k+1}$ with $\tilde{\lambda}^k \in [0,1]$ such that
\begin{align*}
    \sum_{k=0}^N ( f(u^{k+1})-f(u_1^{k+1}+u_2^k), u^{k+1}-u^k) 
    &= \sum_{k=0}^N (\delta^2 F(\tilde{\xi}^k)(u^{k+1}-u^k) , u^{k+1}_2 - u_2^k)\\
    &\leq \sum_{k=0}^N C_2 (\tilde{\xi}^k) 
    \|u^{k+1}-u^k\|_V \|u^{k+1}_2 - u_2^k\|_V\\
    &\leq \sum_{k=0}^N \Bar{C_2}  
    \|u^{k+1}-u^k\|_V \|u^{k+1}_2 - u_2^k\|_V\\
    &\leq \sum_{k=0}^N (\frac{\bar{c}}{2}\|u^{k+1}-u^k\|_V^2 
    + \frac{\Bar{C_2}}{2\bar{c}} \|u^{k+1}_2 - u_2^k\|_V^2).
\end{align*}
Thus, we need 
\begin{align*}
    \sum_{k=0}^N (\frac{\bar{c}}{2}\|u^{k+1}-u^k\|_V^2 
    + \frac{\Bar{C_2}}{2\bar{c}} \|u^{k+1}_2 - u_2^k\|_V^2 )
    \leq (\frac{1-\gamma}{2\alpha_0}-\frac{B}{2}(1+\gamma)) \sum_{k=0}^N \sum_{i=1}^2\|u_i^{k+1}-u_i^k\|^2
    + \sum_{k=0}^N \frac{\bar{c}}{2} \|u^{k+1}-u^k\|_V^2,
\end{align*}
which is equivalent to 
\begin{align}
    \sum_{k=0}^N \frac{\Bar{C_2}}{2\bar{c}} \|u^{k+1}_2 - u_2^k\|_V^2 
    \leq (\frac{1-\gamma}{2\alpha_0}-\frac{B}{2}(1+\gamma)) \sum_{k=0}^N \|u_1^{k+1}-u_1^k\|^2 
    + (\frac{1-\gamma}{2\alpha_0}-\frac{B}{2}(1+\gamma)) \sum_{k=0}^N \|u_2^{k+1}-u_2^k\|^2.
    \label{condition1}
\end{align}
Thus, if $\cfrac{\bar{C}_2}{\bar{c}} \sup_{v_2\in V_{H,2} } \cfrac{\|v_2\|_V^2}{\|v_2\|^2} 
\leq \cfrac{1-\gamma}{\alpha_0}-B(1+\gamma) $, 
then (\ref{condition1}) is satisfied and 
we have 
\begin{align*}
    F(u^{N+1}) + G(u^{N+1}) 
    \leq  F(u^{0}) + G(u^{0}).
\end{align*}
\end{proof}

We remark that if the spaces $V_{H,1}$ and $V_{H,2}$ are constructed appropriately, we have 
\[ \sup_{v_2\in V_{H,2}} \cfrac{ \|\kappa^{\frac{1}{2}} \nabla v_2 \|^2 } { \|v_2\|^2 } = D H^{-2}, \]
where $D$ is a constant and independent on the contrast.
Thus, the time step of our partially explicit scheme scales as the coarse mesh size and does not depend on the contrast \cite{chung2021contrastindependent}.

\section{Numerical Results}
In this section, we first discuss the construction of $V_{H,1}$ and $V_{H,2}$. Then we present numerical results for various $f(u)$ and $g(u)$. We discuss two cases of diffusion operator $f(u)$, the first one is linear
\[ f(u) = -\nabla \cdot ( \kappa \nabla u ) \]
while the second one is nonlinear 
\[ f(u) = -\nabla \cdot ( \kappa (1+u^2) \nabla u). \]
For the reaction term $g(u)$, we use
\begin{align*}
    g(u) = -(10u(u^2-1) + g_0) \quad \text{and} \quad 
    g(u) = -(-\cfrac{10u}{u+2} + g_0), 
\end{align*}
where $g_0$ is the source term. One of the source terms is singular while the other one is smooth.
In all cases, we use a permeability field $\kappa$ which contains many high contrast strikes. We let $\alpha = 0.8$ in all numerical experiments. We take the final time $T=0.05$ and the domain $\Omega = [0,1]^2 \subset \mathbb{R}^2$. In all examples,  the coarse mesh size is $\frac{1}{10}$ and the fine mesh size is $\frac{1}{100}$. The reference solution is obtained using the implicit scheme with fine grid basis functions. We will compare the solution and error for the following three schemes.
\begin{itemize}
    \item the implicit scheme with CEM basis.
    \item the implicit scheme with CEM basis and additional degrees of freedom from $V_{H,2}$. 
    \item the partially explicit scheme with CEM basis and additional degrees of freedom from $V_{H,2}$.
\end{itemize}
In all cases, we use Picard or Newton iterations to solve the nonlinear equations, details of which can be found in \cite{chung2021contrastindependent}. From the numerical results, we find that the partially explicit scheme we proposed obtains similar accuracy as the implicit CEM scheme with additional degrees of freedom.  

\subsection{Construction of $V_{H,1}$ and $V_{H,2}$}
\label{SpaceConstructSection}

In this section, we present a method to build
the two spaces satisfying (\ref{eq:stab_cond}). 
Here, we follow our previous work \cite{chung_partial_expliict21}.
We will first present the CEM finite element space which is used as $V_{H,1}$.
Next we will construct $V_{H,2}$ using eigenvalue problems.
In the following, we let $V(S) = H_0^1(S)$ for a proper 
subset $S\subset \Omega$.

\subsubsection{CEM method}
\label{sec:cem}

In this section, we present the CEM finite element space.
The finite element space is constructed by solving
a constrained energy minimization problem. Let $\mathcal{T}_{H}$
be a coarse grid partition of $\Omega$. For $K_{i}\in\mathcal{T}_{H}$,
we first have to build a collection of auxiliary basis in $V(K_{i})$.
Let $\{\chi_i\}$ be a partition of unity functions corresponding to an overlapping partition of the domain.
We solve the following eigenvalue problem:
\begin{align*} \int_{K_i} \kappa \nabla \psi_j^{(i)} \cdot \nabla v = \lambda_j^{(i)} s_i ( \psi_j^{(i)},v), \quad
\forall v \in V(K_i), \end{align*}
where \begin{align*} s_i(u,v) = \int_{K_i} \tilde{\kappa} u v, \quad
\tilde{\kappa} = \kappa H^{-2} \; \text{or} \; 
\tilde{\kappa} = \kappa \sum_{i}\left|\nabla \chi_{i}\right|^{2}.   \end{align*}  
We then rearrange and gather the $L_i$ eigenfunctions corresponding to the first $L_i$ smallest eigenvalues. 
We define \[ V_{aux}^{(i)}:=\text{span}\{\psi_{j}^{(i)}:\;1\leq j\leq L_{i}\}.  \]
We need to construct a  projection operator $\Pi:L^{2}(\Omega)\mapsto V_{aux}\subset L^{2}(\Omega)$
\[
s(\Pi u,v)=s(u,v),\quad\forall v\in V_{aux}:=\sum_{i=1}^{N_{e}}V_{aux}^{(i)},
\]
with $s(u,v):=\sum_{i=1}^{N_{e}}s_{i}(u|_{K_{i}},v|_{K_{i}})$ and $N_e$ being the number of coarse elements.
We define $K_{i}^{+}$ to be an oversampling domain of $K_{i}$, which is a few coarse blocks larger than $K_i$ \cite{chung2018constraint}. 
For every auxiliary basis $\psi_{j}^{(i)}$, we search for
a local basis function $\phi_{j}^{(i)}\in V(K_{i}^{+})$
such that 
\begin{align*}
a(\phi_{j}^{(i)},v)+s(\mu_{j}^{(i)},v) & =0,\quad\forall v\in V(K_{i}^{+}),\\
s(\phi_{j}^{(i)},\nu) & =s(\psi_{j}^{(i)},\nu),\quad\forall\nu\in V_{aux}(K_{i}^{+}),
\end{align*}
for some $\mu_{j}^{(i)} \in V_{aux}$.
Then the CEM finite element space is 
\begin{align*}
V_{cem} & :=\text{span}\{\phi_{j}^{(i)}:\;1\leq i\leq N_{e},1\leq j\leq L_{i}\}.
\end{align*}
The CEM solution $u_{cem}$ is given by
\begin{align*}
(\cfrac{ \partial u_{cem} }{\partial t} ,v ) + ( f(u_{cem}),v) + (g(u_{cem}) , v ) = 0, \quad \forall v\in V_{cem}.
\end{align*}
Let $\tilde{V}:= \{v\in V: \; \Pi(v) = 0 \}$ and we will construct $V_{H,2}$ in the next section.

\subsubsection{Construction of $V_{H,2}$}





We construct the second space $V_{H,2}$ based on the CEM type finite
element space. For
each coarse element $K_{i}$, we obtain the second type of auxiliary
basis by solving an eigenvalue problem. 
For each $K_i$, after solving
\begin{align}
\label{eq:spectralCEM2}
\int_{K_{i}}\kappa\nabla\xi_{j}^{(i)}\cdot\nabla v & =\gamma_{j}^{(i)}\int_{K_{i}}\xi_{j}^{(i)}v, \;\ \forall v\in V(K_{i})\cap\tilde{V},
\end{align}
($(\xi_{j}^{(i)},\gamma_{j}^{(i)})\in(V(K_{i})\cap\tilde{V})\times\mathbb{R}$), we rearrange and select the first $J_i$ eigenfunctions corresponding to the smallest $J_i$ eigenvalues. 
We define the second auxilliary space $V_{aux,2} := \text{span}\{\xi_j^{(i)} : 1\leq i \leq N_e, 1\leq j\leq J_i \}$.
For each auxiliary basis $\xi_j^{(i)} \in V_{aux,2}$, we define a basis $\zeta_{j}^{(i)} \in V(K_i^+)$ such
that for some $\mu_{j}^{(i),1} \in V_{aux}$, $ \mu_{j}^{(i),2} \in V_{aux,2}$, we have 
\begin{align}
a(\zeta_{j}^{(i)},v)+s(\mu_{j}^{(i),1},v)+ ( \mu_{j}^{(i),2},v) & =0, \quad\forall v\in V(K_i^+), \label{eq:v2a} \\
s(\zeta_{j}^{(i)},\nu) & =0, \quad\forall\nu\in V_{aux}, \label{eq:v2b} \\
(\zeta_{j}^{(i)},\nu) & =( \xi_{j}^{(i)},\nu), \quad\forall\nu\in V_{aux,2}. \label{eq:v2c}
\end{align}
We define $$V_{H,2}=\text{span}\{\zeta_{j}^{(i)}| \; 1\leq i \leq N_e, \;  1 \leq j\leq J_i\}.$$

\subsection{Linear Diffusion Operator $f(u)$}
In this section, we use a linear diffusion operator 
\[ f(u) = - \nabla \cdot (\kappa \nabla u).  \]
After time discretization, Equation (\ref{eq0}) becomes  
\begin{align*}
    \frac{1}{\alpha_0} \sum_{j=0}^k b_{k-j} (u^{j+1}-u^j)  
    - \nabla \cdot (\kappa \nabla u) + g(u) = 0.
\end{align*}
In our first case, the reaction term $g(u) = -(10u(u^2-1) + g_0)$, where $g_0$ is a singular source term.
We set the time step $\Delta t = \frac{T}{4000}$.
In Figure \ref{NRfig1}, we present the permeability field $\kappa$ and the singular source term $g_0$. Figure \ref{NRfig2} shows the implicit fine grid solution (reference solution), implicit CEM solution with additional basis and the partially explicit solution at the final time. The relative $L^2$ error plot and relative energy error plot is presented in Figure \ref{NRfig3}. In all error plots, the blue, red and black curves represent the implicit CEM scheme, the implicit CEM scheme with additional basis and the partially explicit scheme, respectively. From these two error plots in Figure \ref{NRfig3}, we find that there is a considerable improvement in accuracy when we use additional basis. We also find that the error curves for our proposed partially explicit scheme coincide with the error curves for implicit CEM scheme with additional basis, which implies that these two schemes can obtain similar accuracy.
\begin{figure}[H]
\centering
\subfigure{
\includegraphics[width = 6cm]{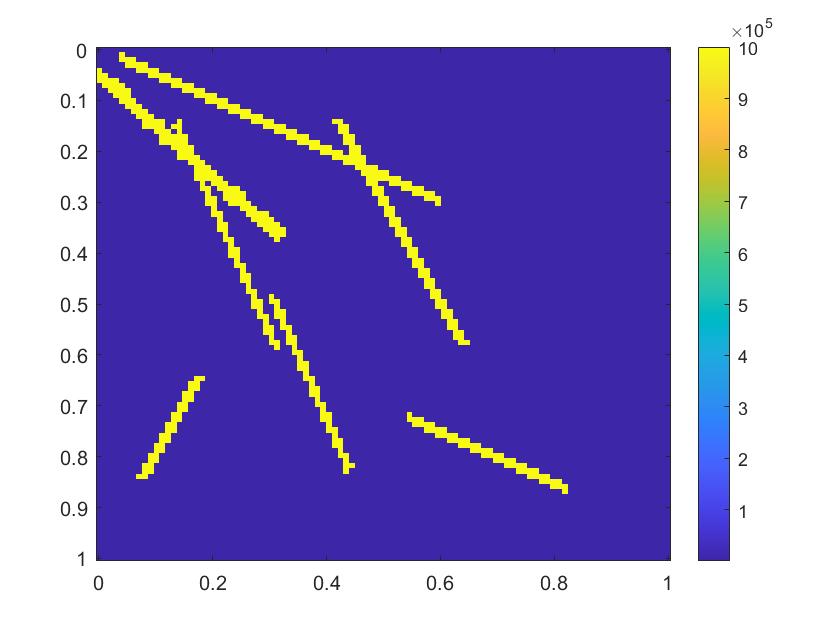}
}
\subfigure{
\includegraphics[width = 6cm]{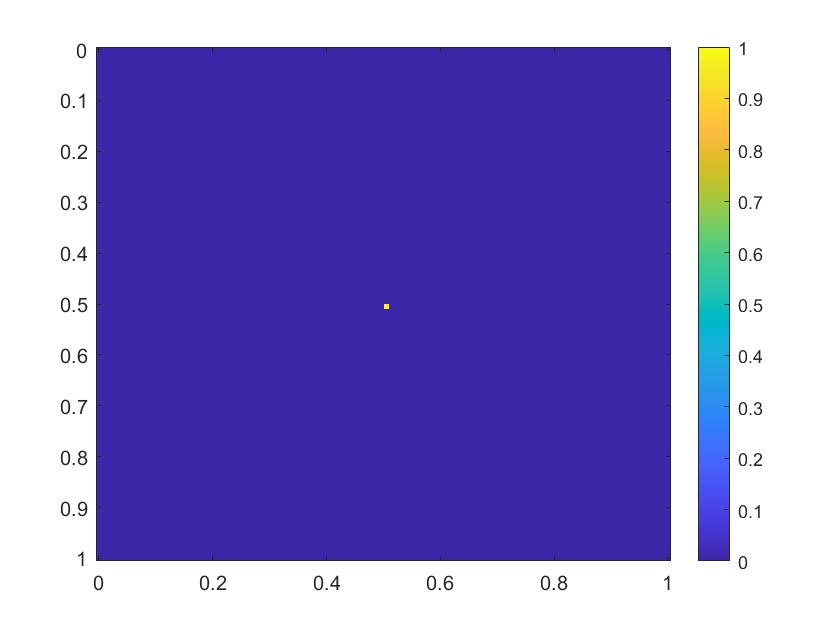}
}
\caption{Left: $\kappa$. Right: $g_0$.}
\label{NRfig1}
\end{figure}
\begin{figure}[H]
\centering
\subfigure{
\includegraphics[width = 5cm]{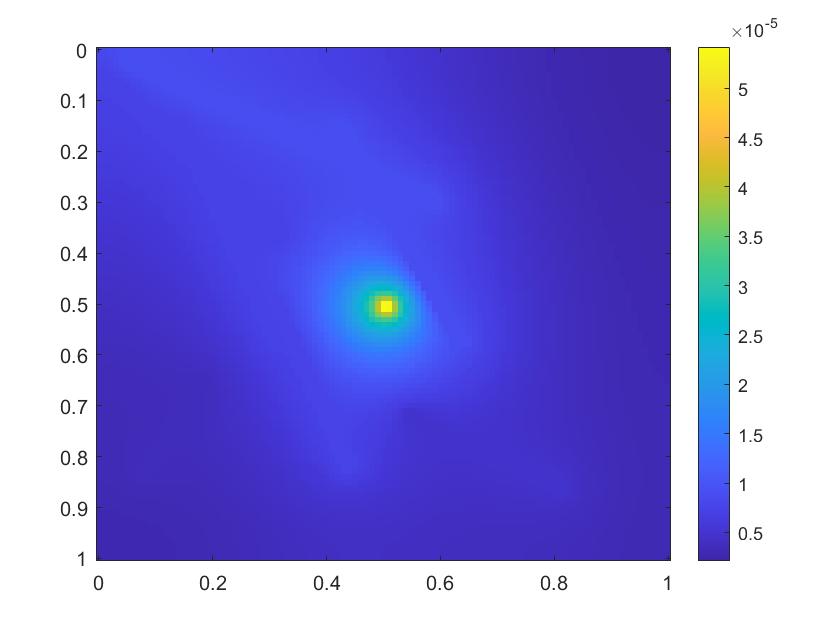}
}
\subfigure{
\includegraphics[width = 5cm]{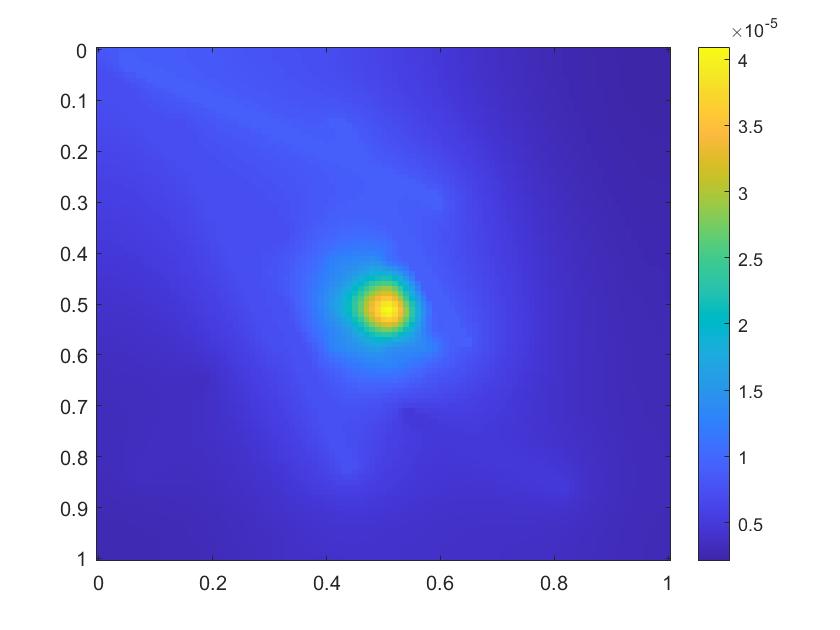}
}
\subfigure{
\includegraphics[width = 5cm]{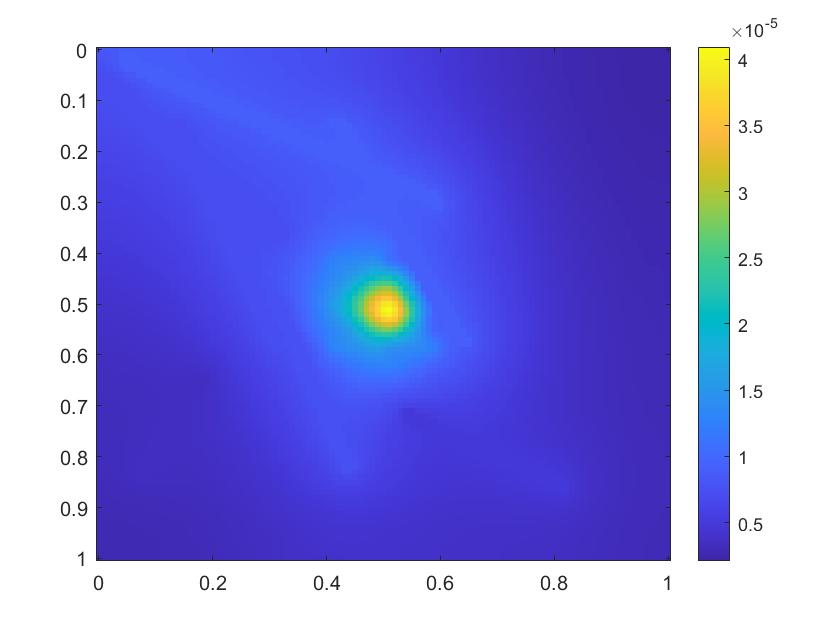}
}
\caption{Left: Implicit fine grid solution at $t=T$. 
Middle: Implicit CEM solution with additional basis at $t=T$. 
Right: Partially explicit solution at $t=T$.}
\label{NRfig2}
\end{figure}
\begin{figure}[H]
\centering
\subfigure{
\includegraphics[width = 6cm]{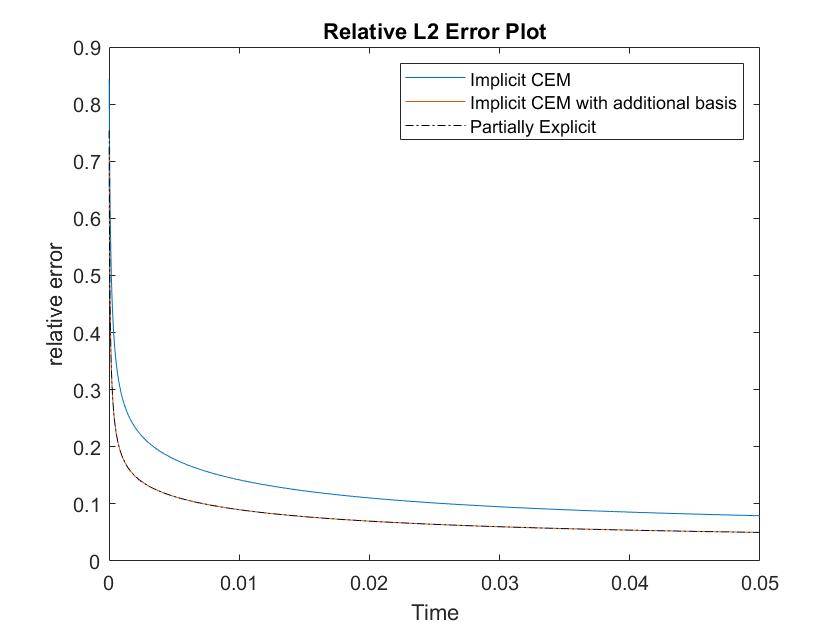}
}
\subfigure{
\includegraphics[width = 6cm]{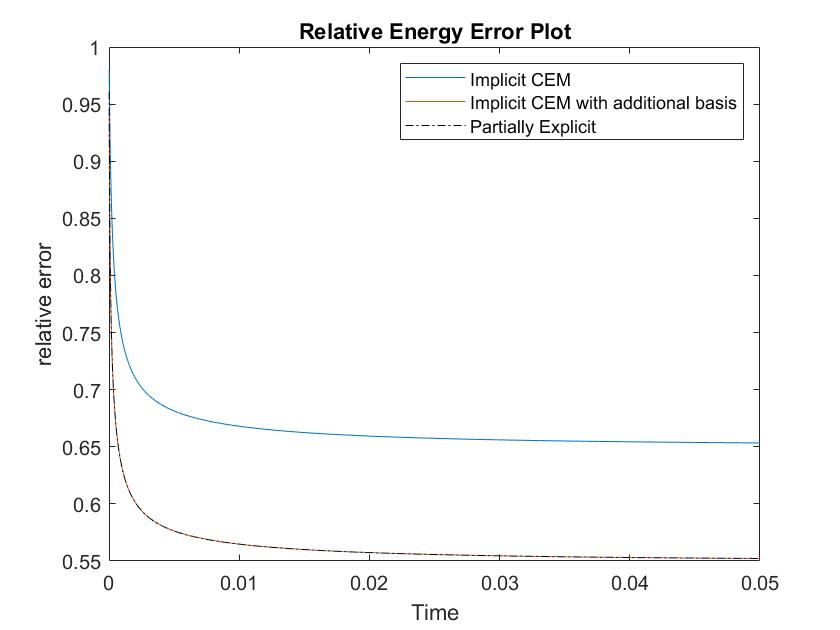}
}
\caption{Left: Relative $L^2$ error. 
Right: Relative energy error.}
\label{NRfig3}
\end{figure}

In this case, we consider $g(u) = -(10u(u^2-1) + g_0)$ and we use a smooth source term $g_0$. 
We set $\Delta t = \frac{T}{4000}$. The permeability field $\kappa$ and the source term $g_0$ is presented in Figure \ref{NRfig4}. The reference solution, implicit CEM solution with additional basis and the partially explicit solution at $t=T$ are shown in Figure \ref{NRfig5}. The relative $L^2$ and energy error plots are presented in Figure \ref{NRfig6}. In this case, the errors for three schemes are comparable. The improvement from the additional basis is not as obvious as our first case. We can find that the partially explicit scheme can achieve similar accuracy as the implicit CEM scheme with additional basis. 
\begin{figure}[H]
\centering
\subfigure{
\includegraphics[width = 6cm]{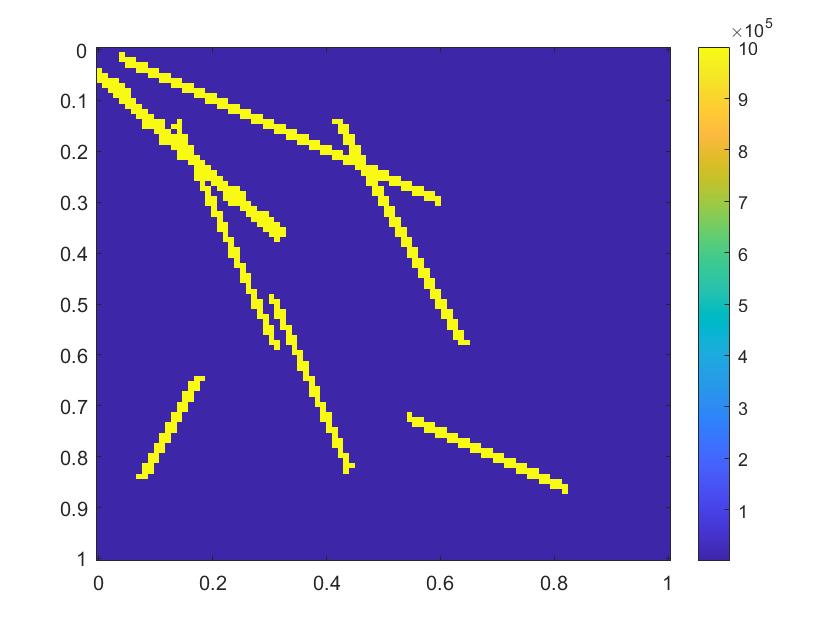}
}
\subfigure{
\includegraphics[width = 6cm]{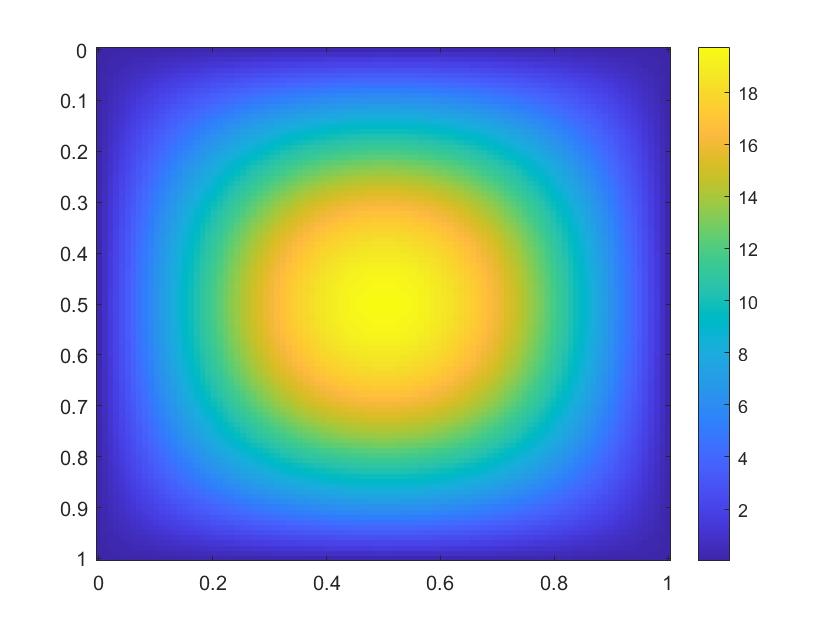}
}
\caption{Left: $\kappa$. Right: $g_0$.}
\label{NRfig4}
\end{figure}

\begin{figure}[H]
\centering
\subfigure{
\includegraphics[width = 5cm]{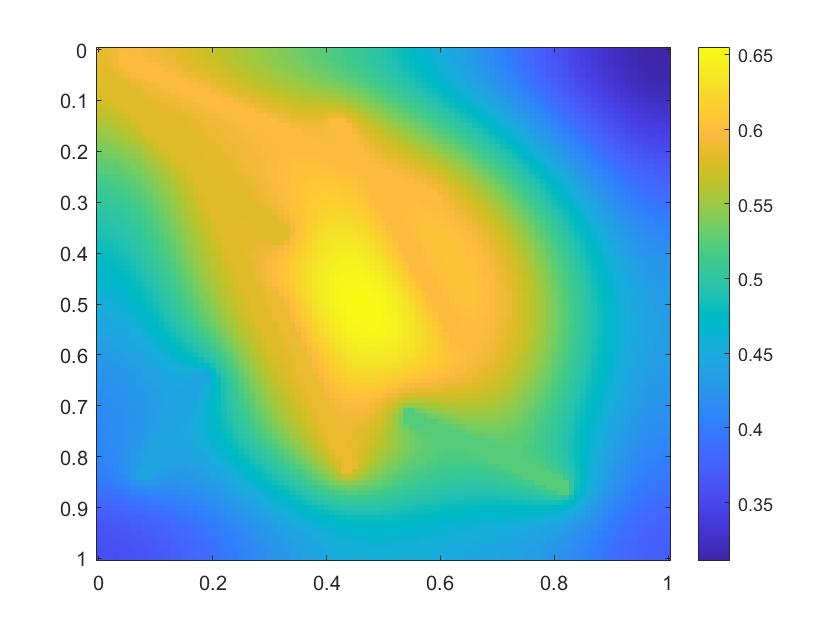}
}
\subfigure{
\includegraphics[width = 5cm]{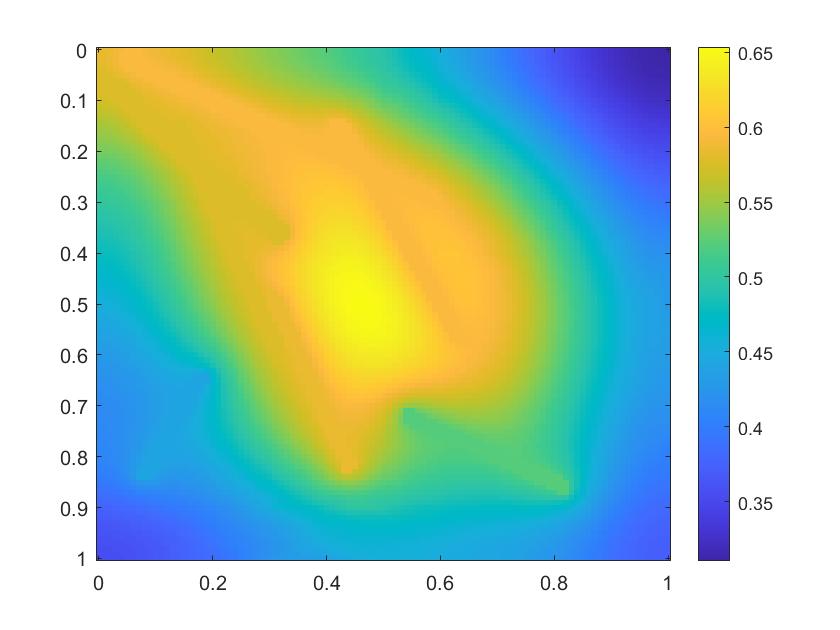}
}
\subfigure{
\includegraphics[width = 5cm]{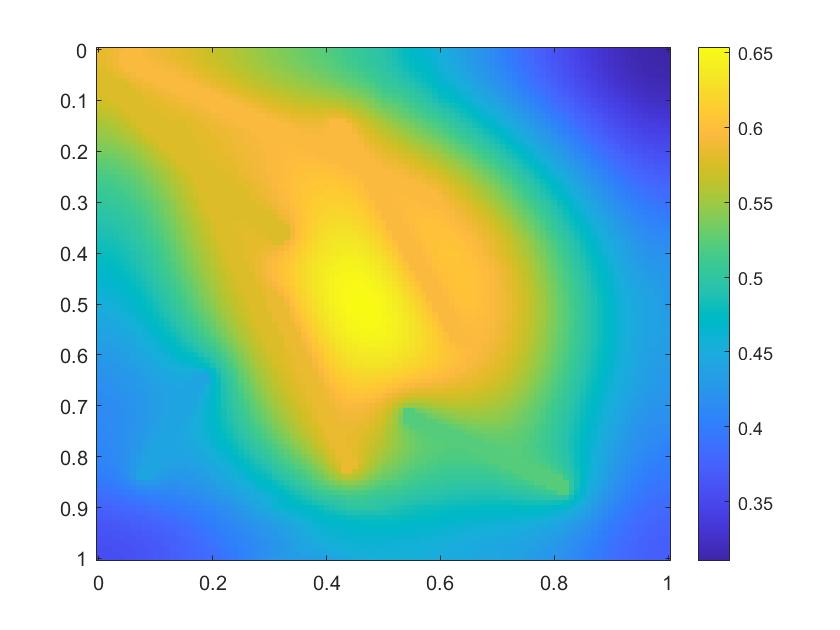}
}
\caption{Left: Implicit fine grid solution at $t=T$. 
Middle: Implicit CEM solution with additional basis at $t=T$. 
Right: Partially explicit solution at $t=T$.}
\label{NRfig5}
\end{figure}

\begin{figure}[H]
\centering
\subfigure{
\includegraphics[width = 6cm]{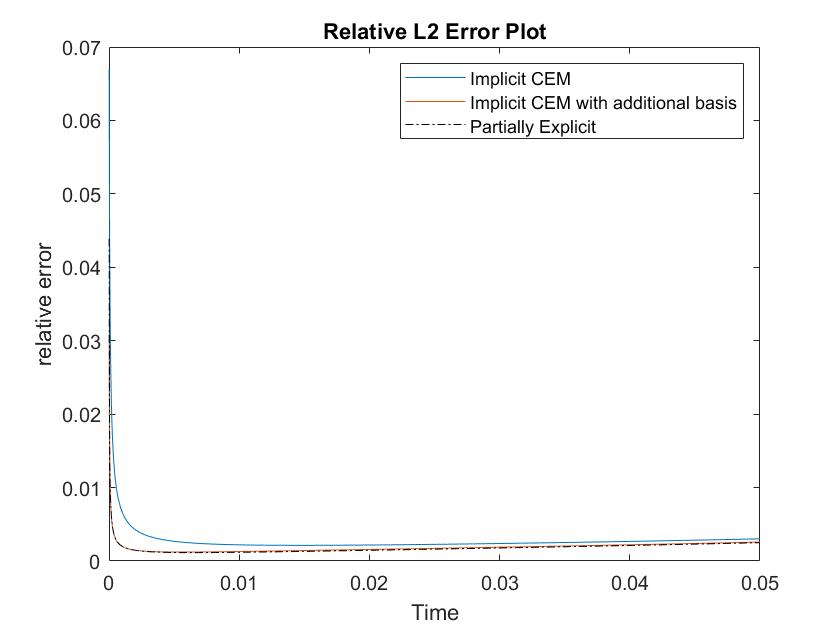}
}
\subfigure{
\includegraphics[width = 6cm]{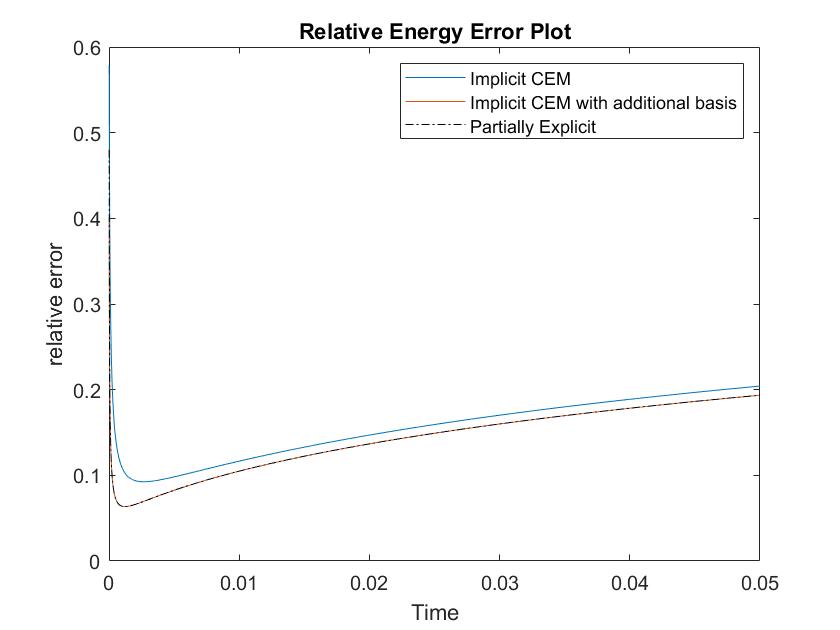}
}
\caption{Left: Relative $L^2$ error. 
Right: Relative energy error.}
\label{NRfig6}
\end{figure}

In this case, we use a new reaction term
\[g(u) =-(-\cfrac{10u}{u+2} +g_0). \]
We set $\Delta t = \frac{T}{4000}$. Figure \ref{NRfig7} shows $\kappa$ and $g_0$. We present the implicit fine grid solution (reference solution) at the final time, implicit CEM solution with additional basis at the final time and partially explicit solution at the final time in Figure \ref{NRfig8}. The relative $L^2$ error plot and relative energy error plot are presented in Figure \ref{NRfig9}. From these two relative error plots, we observe that the errors for partially explicit scheme resemble the errors for implicit CEM scheme with additional basis.
\begin{figure}[H]
\centering
\subfigure{
\includegraphics[width = 5cm]{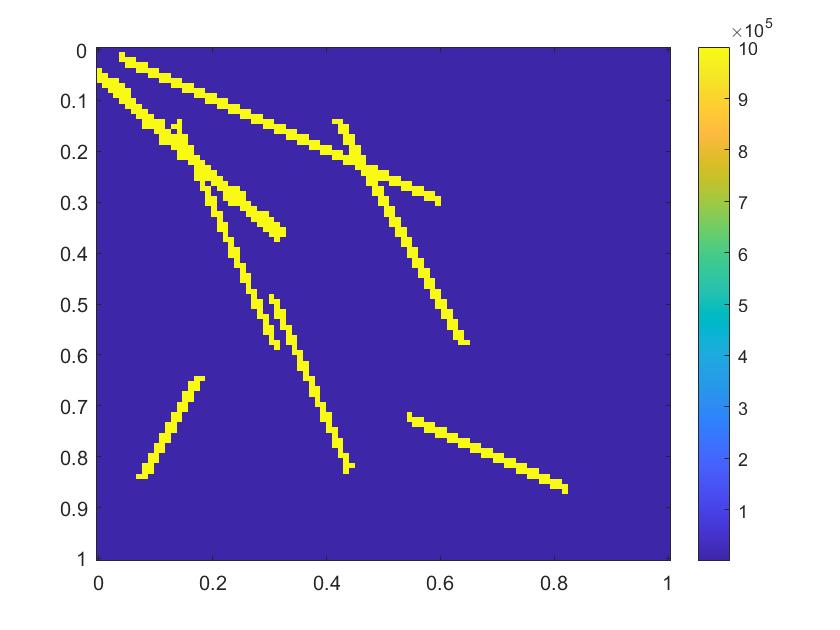}
}
\subfigure{
\includegraphics[width = 5cm]{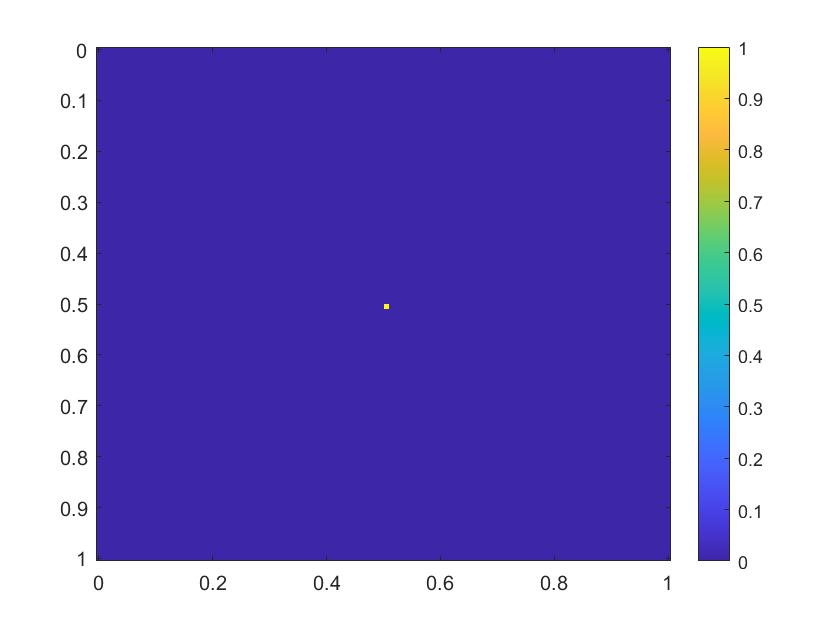}
}
\caption{Left: $\kappa$. Right: $g_0$.}
\label{NRfig7}
\end{figure}

\begin{figure}[H]
\centering
\subfigure{
\includegraphics[width = 5cm]{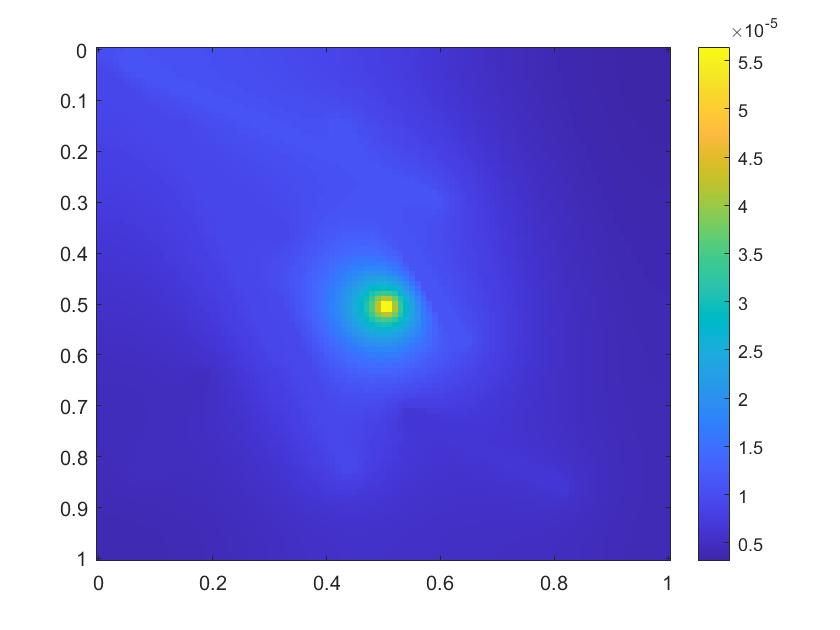}
}
\subfigure{
\includegraphics[width = 5cm]{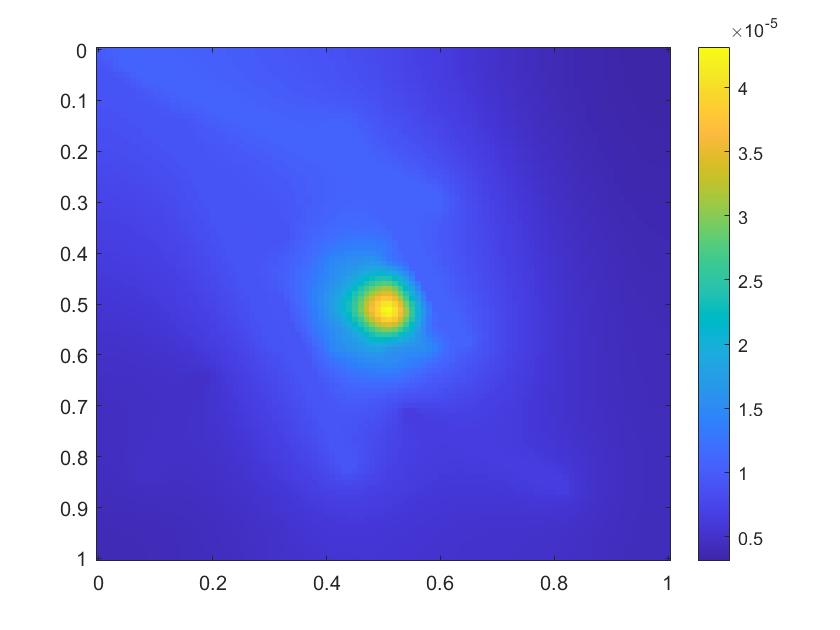}
}
\subfigure{
\includegraphics[width = 5cm]{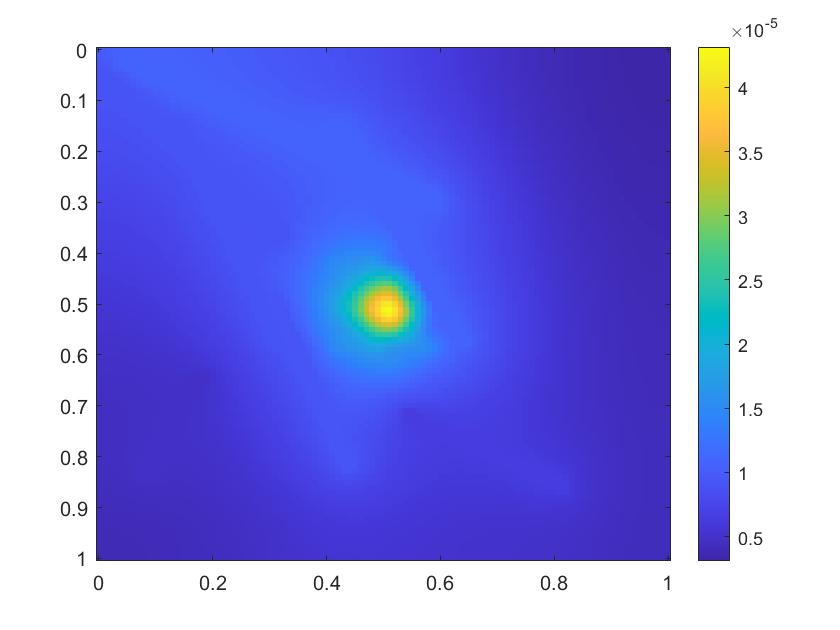}
}
\caption{Left: Implicit fine grid solution at $t=T$. 
Middle: Implicit CEM solution with additional basis at $t=T$. 
Right: Partially explicit solution at $t=T$.}
\label{NRfig8}
\end{figure}

\begin{figure}[H]
\centering
\subfigure{
\includegraphics[width = 6cm]{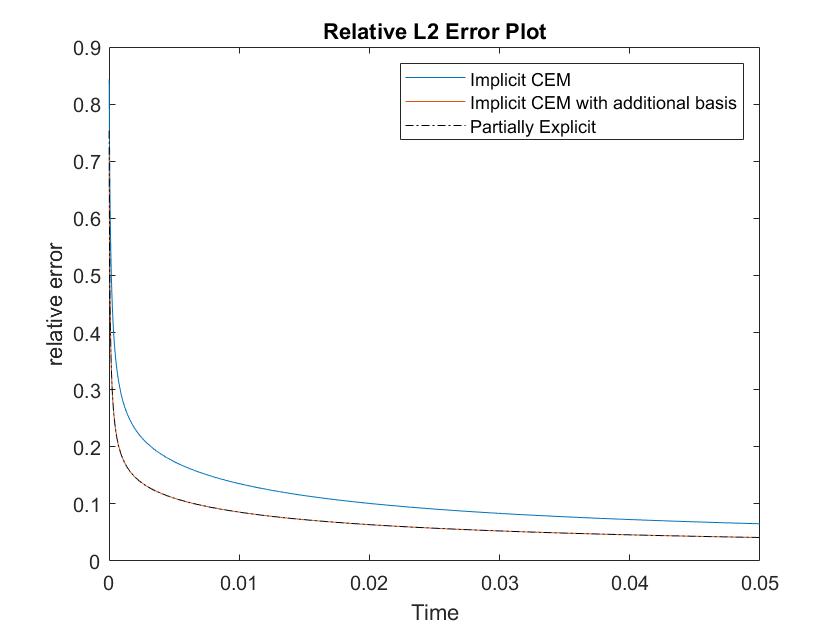}
}
\subfigure{
\includegraphics[width = 6cm]{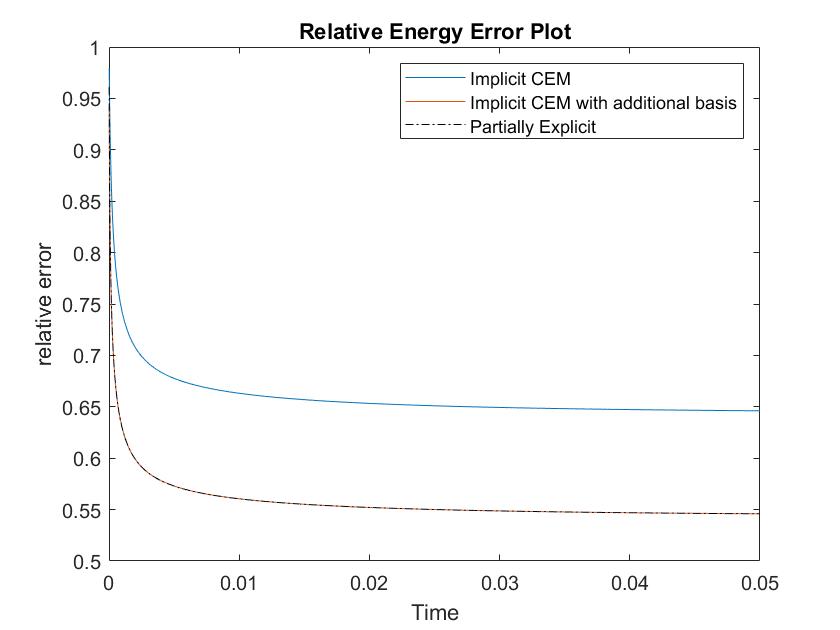}
}
\caption{Left: Relative $L^2$ error. 
Right: Relative energy error.}
\label{NRfig9}
\end{figure}

In this case, we let $g(u) =-(-\cfrac{10u}{u+2}+g_0)$ and we use a smooth source term $g_0$.
We set $\Delta t = \frac{T}{4000}$. The permeability field $\kappa$ and source term $g_0$ are presented in Figure \ref{NRfig10}. The implicit fine grid solution at $t=T$, the implicit CEM solution with additional basis at $t=T$ and the partially explicit solution at $t=T$ are shown in Figure \ref{NRfig11}. We present the relative $L^2$ error plot and the relative energy error plot in Figure \ref{NRfig12}. There is improvement when we add more basis from $V_{H,2}$, though not very considerable. We can also find that the accuracy for the partially explicit scheme is similar to that of the implicit scheme with additional basis. 
\begin{figure}[H]
\centering
\subfigure{
\includegraphics[width = 5cm]{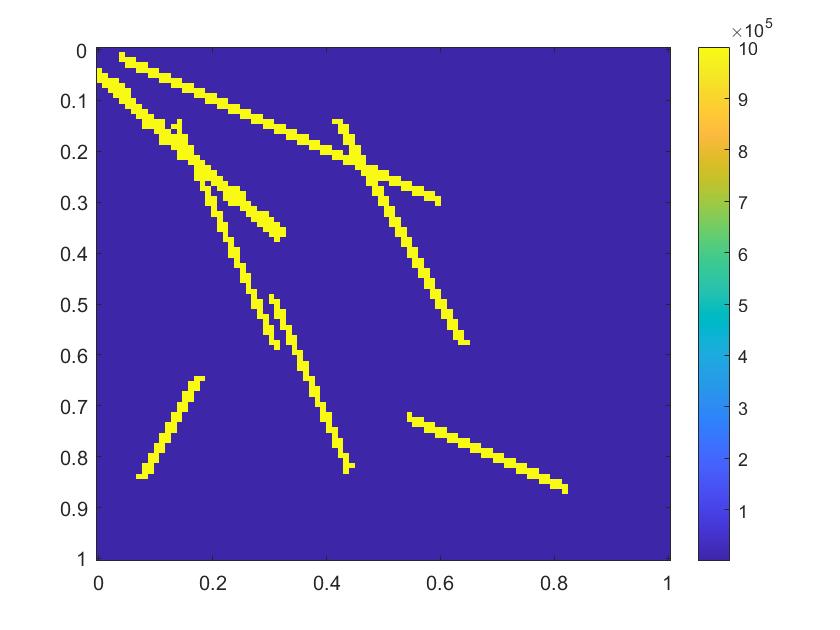}
}
\subfigure{
\includegraphics[width = 5cm]{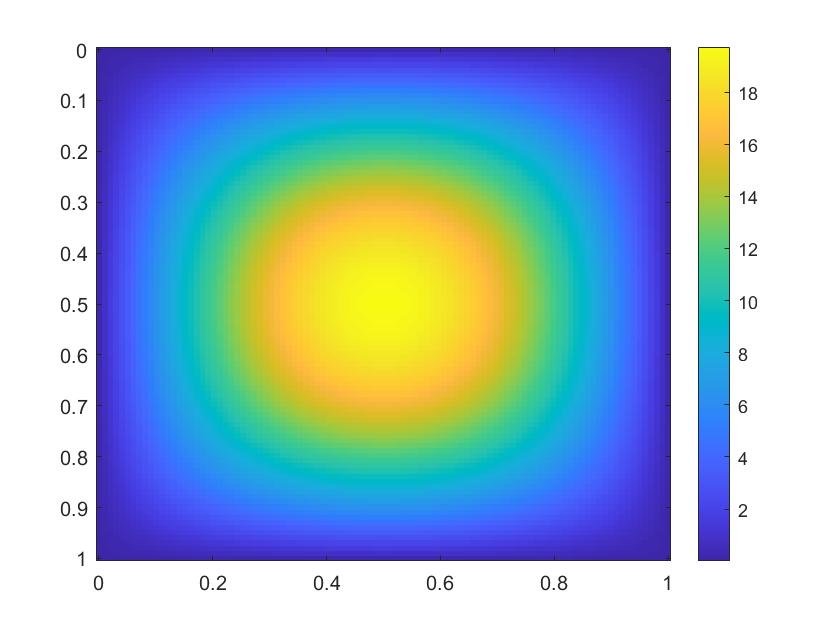}
}
\caption{Left: $\kappa$. Right: $g_0$.}
\label{NRfig10}
\end{figure}

\begin{figure}[H]
\centering
\subfigure{
\includegraphics[width = 5cm]{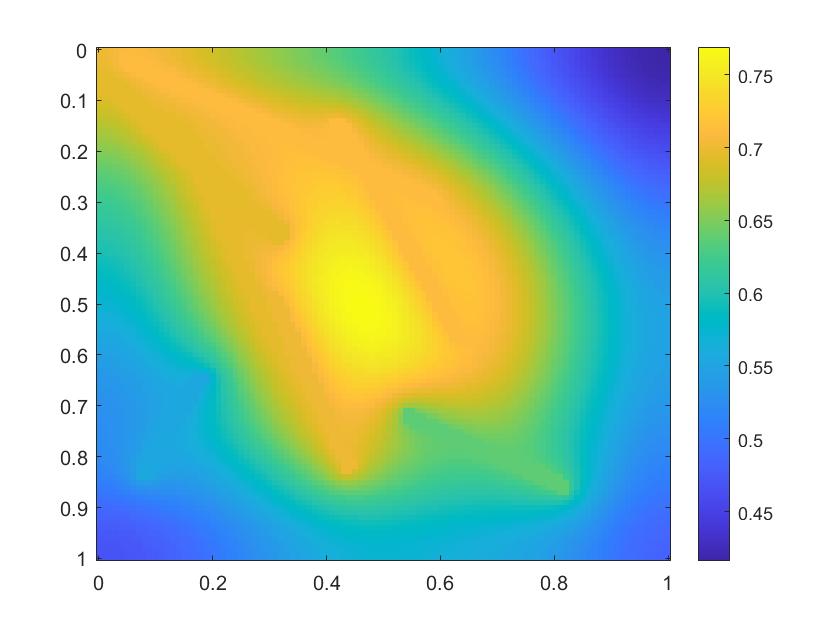}
}
\subfigure{
\includegraphics[width = 5cm]{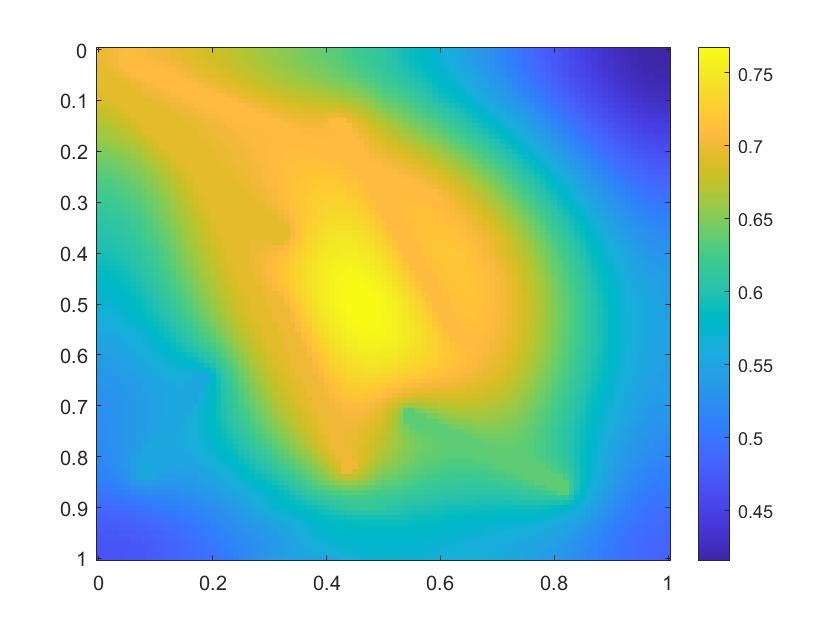}
}
\subfigure{
\includegraphics[width = 5cm]{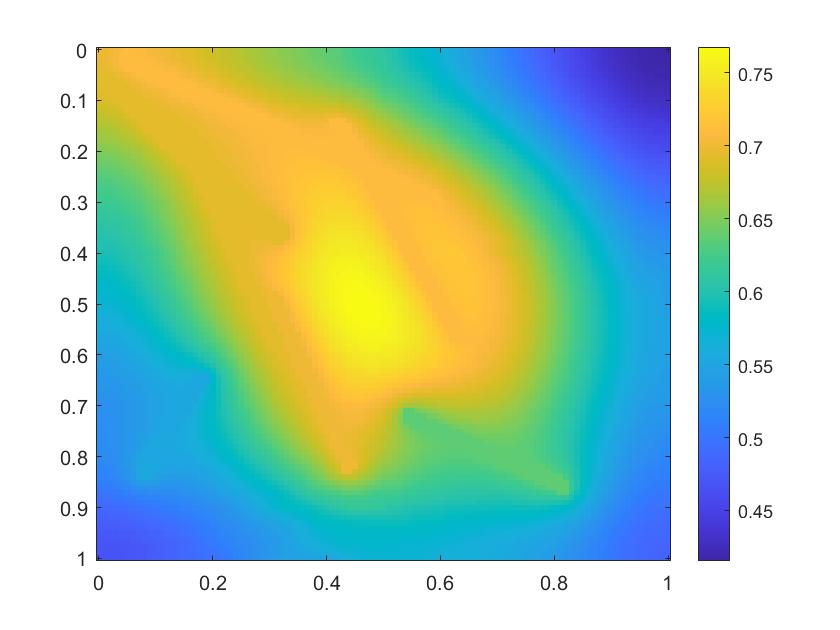}
}
\caption{Left: Implicit fine grid solution at $t=T$. 
Middle: Implicit CEM solution with additional basis at $t=T$. 
Right: Partially explicit solution at $t=T$.}
\label{NRfig11}
\end{figure}

\begin{figure}[H]
\centering
\subfigure{
\includegraphics[width = 6cm]{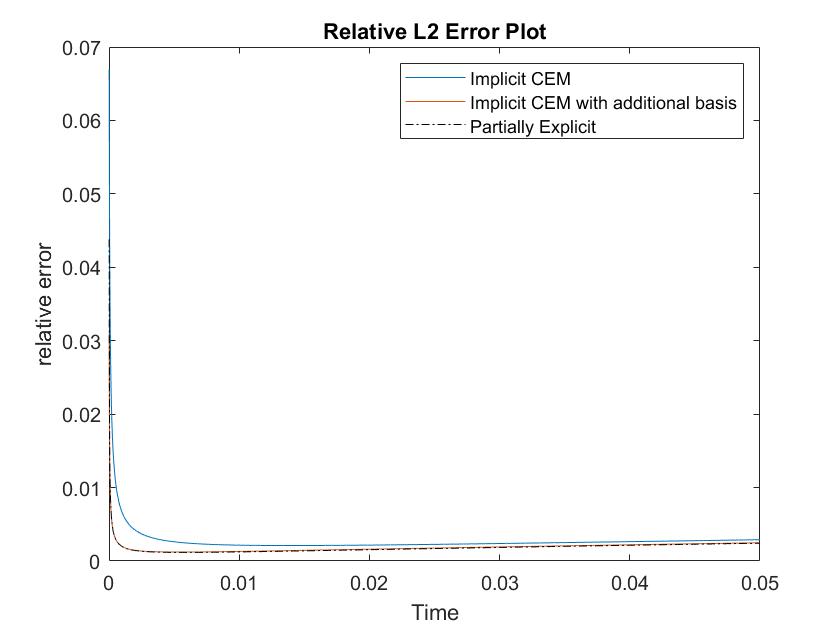}
}
\subfigure{
\includegraphics[width = 6cm]{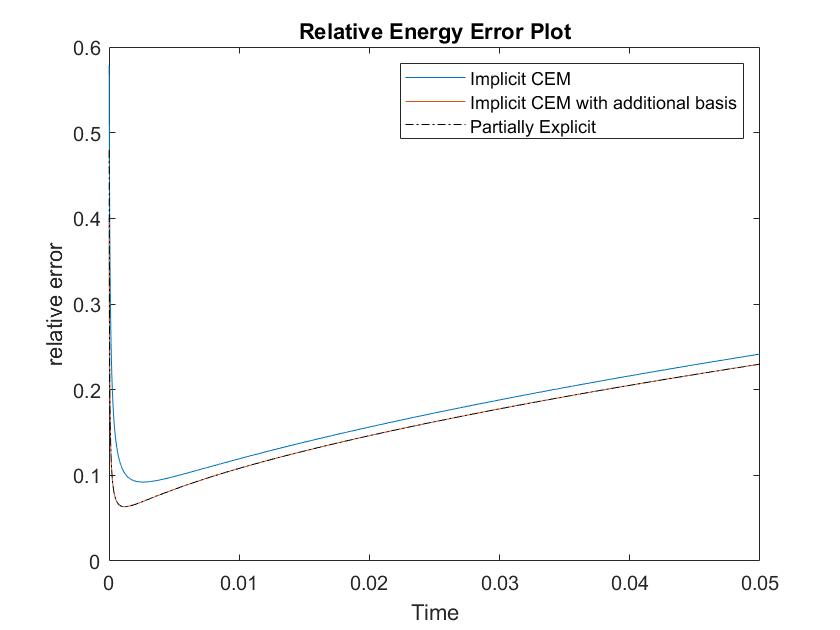}
}
\caption{Left: Relative $L^2$ error. 
Right: Relative energy error.}
\label{NRfig12}
\end{figure}

\subsection{Nonlinear Diffusion Operator $f(u)$}
In this section, we let 
\[f(u) = -\nabla\cdot (\kappa (1+u^2)\nabla u)\] 
and $g(u) = -(10u\cdot(u^2-1) + g_0 )$. 
We set $\Delta t = \frac{T}{8000}$ and $g_0$ is a smooth source term.
We present $\kappa$ and $g_0$ in Figure \ref{NRfig13}. The implicit fine grid solution at $t=T$, implicit CEM solution with additional basis at $t=T$ and partially explicit solution $t=T$ are presented in Figure \ref{NRfig14}. The relative $L^2$ error plot and the relative energy error plot are shown in Figure \ref{NRfig15}. From these plots, we can conclude that our proposed partially explicit scheme also works for nonlinear diffusion operators. The errors for implicit CEM scheme with additional basis are similar to the errors for partially explicit scheme, which implies that these two schemes have similar accuracy.
\begin{figure}[H]
\centering
\subfigure{
\includegraphics[width = 5cm]{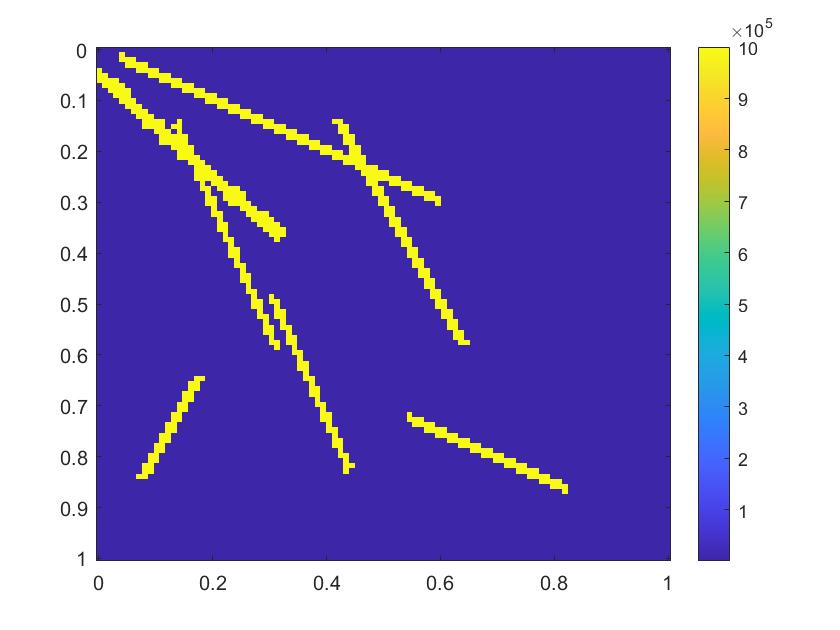}
}
\subfigure{
\includegraphics[width = 5cm]{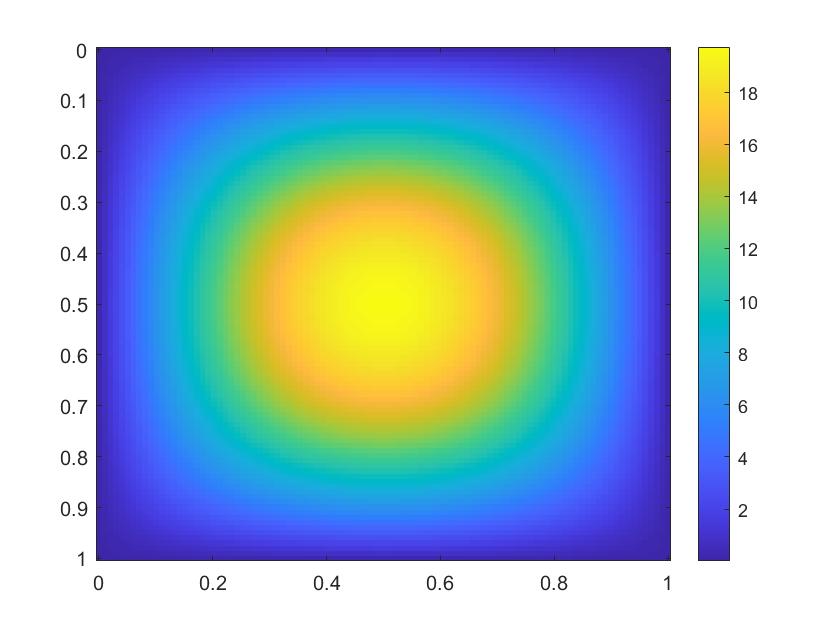}
}
\caption{Left: $\kappa$. Right: $g_0$.}
\label{NRfig13}
\end{figure}
\begin{figure}[H]
\centering
\subfigure{
\includegraphics[width = 5cm]{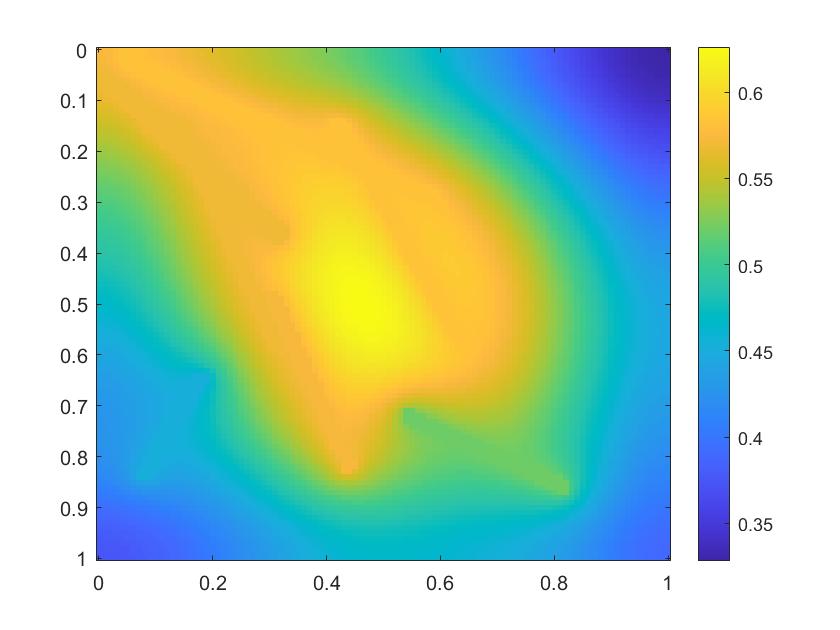}
}
\subfigure{
\includegraphics[width = 5cm]{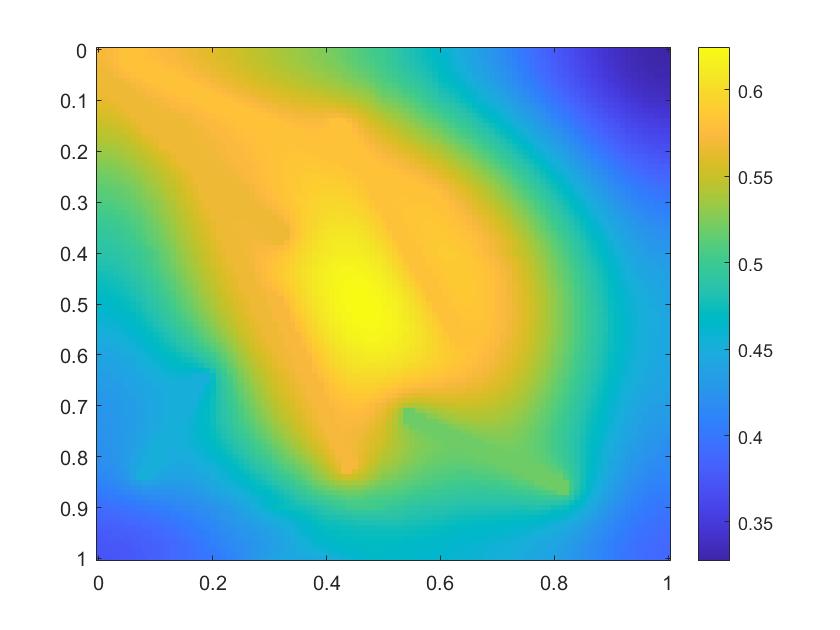}
}
\subfigure{
\includegraphics[width = 5cm]{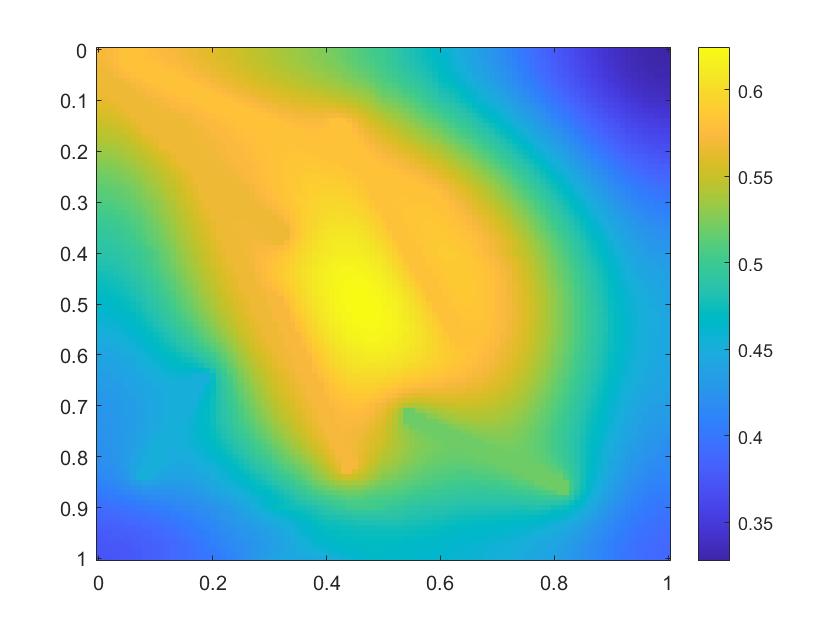}
}
\caption{Left: Implicit fine grid solution at $t=T$. 
Middle: Implicit CEM solution with additional basis at $t=T$. 
Right: Partially explicit solution at $t=T$.}
\label{NRfig14}
\end{figure}

\begin{figure}[H]
\centering
\subfigure{
\includegraphics[width = 6cm]{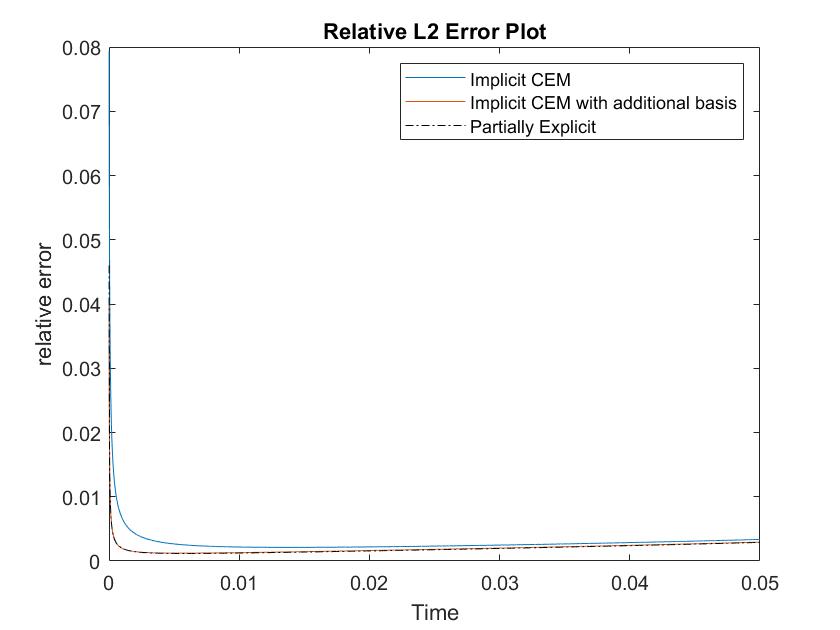}
}
\subfigure{
\includegraphics[width = 6cm]{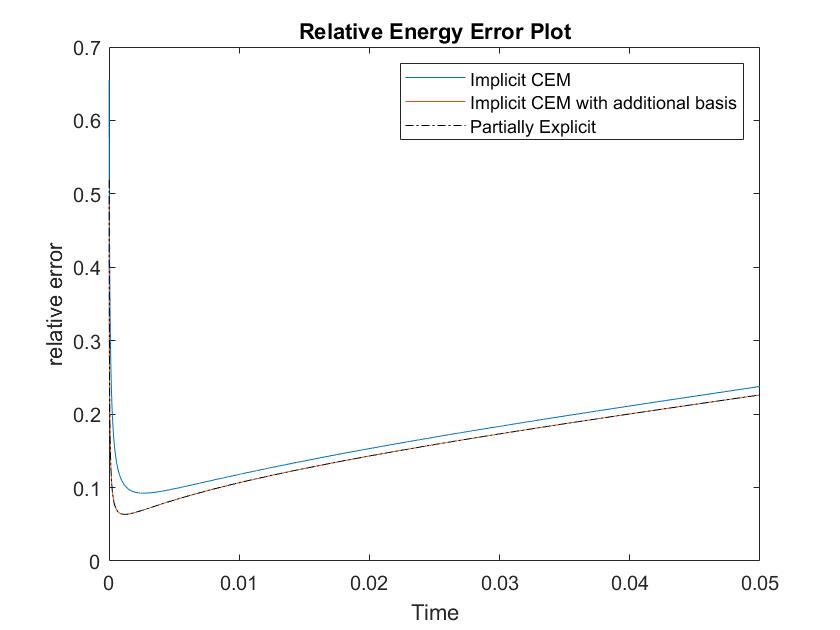}
}
\caption{Left: Relative $L^2$ error. 
Right: Relative energy error.}
\label{NRfig15}
\end{figure}

\section{Conclusions}
In this paper, we propose a partially explicit scheme to solve the time fractional nonlinear diffusion reaction equation. This work continues the earlier works on linear problems. The explicit scheme requires extremely small time step because of high contrast and time fractional power.  We formulate a condition on the spaces for the partially explicit scheme to be stable. We introduce a way to split $V_H$ and construct $V_{H,1}$ and $V_{H,2}$ which satisfy the stability condition. We find that the time step in the partially explicit scheme scales as the coarse mesh size, which provides computational savings. The main advantage 
of our scheme is that it handles additional degrees of freedom
(those beyond few coarse-grid degrees of freedom) explicitly for
nonlinear diffusion and reaction terms.
From the numerical experiments, we conclude that the partially explicit scheme can obtain similar accuracy as the fully implicit CEM scheme with additional basis.

\bibliographystyle{abbrv}
\bibliography{references,references4,references1,references2,references3,decSol,referencesFD}

\end{document}